\documentclass[a4paper, 12pt]{amsart}
\usepackage{graphics,color,hyperref,verbatim}
\usepackage{amsmath}
\usepackage{amsfonts}
\usepackage{amssymb}
\usepackage{amsthm}
\usepackage{graphicx}
\usepackage{psfrag}
\usepackage{marvosym}
\usepackage{mathrsfs}
\usepackage{enumitem}
\usepackage{cleveref}

\usepackage[
	a4paper,
	top=2.3cm, left=2.5cm, right=2.5cm, bottom=2.8cm,
	heightrounded,
	centering
]{geometry}

\usepackage{mathptmx}

\usepackage{multicol} 
\usepackage[usenames,dvipsnames]{xcolor} 
\usepackage{tikz, tikz-3dplot, pgfplots}
\usepackage{tikz-cd}
\usetikzlibrary[positioning,patterns] 
\usetikzlibrary{matrix,arrows,decorations.pathmorphing}

\definecolor{light-gray}{gray}{0.7}

\usepackage{ytableau}

\usepackage{calligra}
\usepackage{mathrsfs}
\DeclareMathAlphabet{\mathcalligra}{T1}{calligra}{m}{n}
\DeclareFontShape{T1}{calligra}{m}{n}{<->s*[1.5]callig15}{}

\usepackage{newtxmath}
\usepackage[T1]{fontenc}

\newtheorem{theorem}{Theorem}[section]

\newtheorem{lemma}[theorem]{Lemma}
\newtheorem{proposition}[theorem]{Proposition}

\newtheorem{corollary}[theorem]{Corollary}

\theoremstyle{definition}
\newtheorem{definition}[theorem]{Definition}
\newtheorem{construction}[theorem]{Construction}
\newtheorem{example}[theorem]{Example}

\newtheorem{remark}[theorem]{Remark}
\newtheorem{theorem-definition}[theorem]{Theorem-Definition}
\newtheorem{lemma-definition}[theorem]{Lemma-Definition}

\newtheorem{variant}[theorem]{Variant}

\numberwithin{equation}{section}

\newcommand{\CC} {\mathbb{C}}

\newcommand{\RR} {\mathbb{R}}

\newcommand {\shQ} {\mathcal{Q}}

\newcommand {\shS} {\mathcal{S}}

\newcommand {\sE} {\mathscr{E}}
\newcommand {\sF} {\mathscr{F}}
\newcommand {\sG} {\mathscr{G}}
\newcommand {\sH} {\mathscr{H}}

\newcommand {\sO} {\mathscr{O}}

\newcommand {\bG} {\mathbf{G}}

\usepackage{calligra}
\usepackage{mathrsfs}
\DeclareMathAlphabet{\mathcalligra}{T1}{calligra}{m}{n}
\DeclareFontShape{T1}{calligra}{m}{n}{<->s*[1.5]callig15}{}

\newcommand{\blank}{\underline{\hphantom{A}}}

\newcommand {\D} {\operatorname{D}}

\newcommand {\Ext} {\operatorname{Ext}}
\newcommand{\sExt}{\mathscr{E} \kern -1pt xt}

\newcommand {\Hom} {\operatorname{Hom}}
\newcommand {\sHom}{\mathscr{H}\kern-5pt\mathcalligra{om}}

\newcommand {\Spec} {\operatorname{Spec}}

\newcommand{\sTor}{\mathscr{T} \kern -3pt or}

\newcommand{\Quot} {\operatorname{Quot}}
\newcommand{\Hilb} {\operatorname{Hilb}}

\newcommand{\Dqc}{\mathrm{D}_{\mathrm{qc}}}

\newcommand{\Perf} {\mathrm{Perf}}

\newcommand{\cn}{\mathrm{cn}}

\newcommand{\Ind}{\operatorname{Ind}}

\newcommand{\Map}{\mathrm{Map}}

\newcommand{\CAlgDelta}{\operatorname{SCR}}
\newcommand{\Mod}{\operatorname{Mod}}

\newcommand {\perf}{\mathrm{perf}}

\makeatletter
\newcommand*\bigcdot{\mathpalette\bigcdot@{.5}}
\newcommand*\bigcdot@[2]{\mathbin{\vcenter{\hbox{\scalebox{#2}{$\m@th#1\bullet$}}}}}
\makeatother

\makeatletter
\newcommand*\bigdot{\mathpalette\bigdot@{.5}}
\newcommand*\bigdot@[2]{\mathbin{\hbox{\scalebox{#2}{$\m@th#1\bullet$}}}}
\makeatother

\title[]{Abel maps for integral curves via a derived perspective}
\author[]{Qingyuan Jiang}
\address{Department of Mathematics,
The Hong Kong University of Science and Technology, Clearwater Bay, Kowloon, Hong Kong.} 
\email{jiangqy@ust.hk}

\begin{document}
\maketitle

\begin{abstract}  
We develop a general framework for Abel maps associated with a family $X/S$ of integral curves using derived algebraic geometry.  For compactified Picard schemes, our approach yields relative quasi-smooth derived enhancements of the Quot schemes $\mathrm{Quot}_{\omega/X/S}^d$ and, in the Gorenstein case, of the Hilbert schemes of points $\mathrm{Hilb}_{X/S}^d$ on $X/S$. These constructions naturally generalize to higher rank torsion-free sheaves and their coherent systems.

We obtain unified semiorthogonal decompositions for the derived categories of these derived moduli spaces, broadly extending previous results for symmetric powers, varieties of linear series, and Thaddeus pairs to torsion-free sheaves on integral curves.

Central to our approach are two novel tools of independent interest: the $\mathcal{Q}$-complex, a derived generalization of Grothendieck's $Q$-module and the Altman–Kleiman $H$-module, and a derived theory of moduli of extensions that extends Lange's classical framework.
\end{abstract}

\tableofcontents

\section{Introduction}

Let $X \to S$ be a flat, projective map of locally Noetherian schemes, whose geometric fibers are integral curves of arithmetic genus $p_a \ge 1$.  For simplicity, we assume that the smooth locus of $X \to S$ admits a section $x \colon S \to X$.

The (rigidified) Picard functor $\mathrm{Pic}_{X/S}$ assigns to each $S$-scheme $T$ the isomorphism classes
\[
\mathrm{Pic}_{X/S}(T) = \left\{ (\mathscr{L}, u) \mid \text{$\mathscr{L}$ is a line bundle on $X_T: = X \times_S T$},\ u \colon \mathscr{O}_T \xrightarrow{\simeq} x_T^* (\mathscr{L}) \right\} / \simeq.
\]
This functor is represented by a disjoint union $\coprod_{d \in \mathbb{Z}} \mathrm{Pic}_{X/S}^d$ of $S$-schemes, where each $\mathrm{Pic}_{X/S}^d$ parametrizes line bundles of degree $d$ on the fibers of $X/S$ (\cite{Gro,Kl05}). However, the Picard scheme $\mathrm{Pic}_{X/S}^d$ is generally not proper over $S$ unless $X/S$ is smooth.

Altman and Kleiman \cite{AK79,AK80} constructed a compactification $\overline{\mathrm{Pic}}_{X/S}$ of $\mathrm{Pic}_{X/S}$, which assigns to each $S$-scheme $T$ the set of isomorphism classes
\[
\overline{\mathrm{Pic}}_{X/S}(T) = \left\{ (\mathscr{I}, u) \mid
\begin{array}{l}
\text{$\mathscr{I}$ is a $T$-flat coherent sheaf on $X_T$ that is rank $1$,}\\
\text{torsion-free on fibers of $X_T/T$,}~u \colon \mathscr{O}_T \xrightarrow{\simeq} x_T^*(\mathscr{I})
\end{array}
\right\} / \simeq.
\]
By \cite[Thm. 3.4(3)]{AK79} and \cite[Thm. 8.5]{AK80}, this functor is represented by a disjoint union $\coprod_{d \in \mathbb{Z}} \overline{\mathrm{Pic}}_{X/S}^d$ of projective $S$-schemes, where each $\overline{\mathrm{Pic}}_{X/S}^d$ parametrizes rank $1$, torsion-free sheaves of degree $d$ on the fibers of $X/S$. Each $\overline{\mathrm{Pic}}_{X/S}^d$ contains $\mathrm{Pic}_{X/S}^d$ as an open subscheme.

A key tool in studying these compactified Picard schemes is the \emph{Abel map} \cite[(8.2)]{AK80}
\[
A_\omega^d \colon \Quot_{\omega/X/S}^{2p_a-2-d} \to \overline{\mathrm{Pic}}_{X/S}^{d}.
\]
Here, for an integer $m$, $\Quot_{\omega/X/S}^{m}$ denotes the Quot scheme parametrizing degree $m$ torsion quotients of the relative dualizing sheaf $\omega$ on fibers of $X/S$. Concretely, for any $S$-scheme $T$,
\[
\Quot_{\omega/X/S}^{m}(T) = \left\{ \omega_{X_T/T} \twoheadrightarrow \mathscr{G} \mid \mathscr{G} \text{ is $T$-flat, rank $0$, degree $m$ on fibers of $X_T/T$} \right\} / \simeq,
\]
and the Abel map sends such a quotient to its kernel. The fiber of $A_\omega^d$ over a point $[I] \in  \overline{\mathrm{Pic}}^{d}(X_s)$ is the projective space $\mathbb{P}(\operatorname{Hom}_{X_s}(I, \omega_{X_s})^\vee) \simeq \mathbb{P}(H^1(X_s, I))$, where $X_s:=X \times_S \{s\}$.

It is often useful to consider the ``modified" Abel map (\cite[Definition 5.0.8]{Kas13})
\[
A^d \colon \Quot_{\omega/X/S}^d \to \overline{\mathrm{Pic}}_{X/S}^d,
\]
obtained by composing $A_\omega^{2p_a-2-d}$ with the duality isomorphism $\sHom(-, \omega)\colon \overline{\mathrm{Pic}}_{X/S}^{2p_a-2-d} \xrightarrow{\sim} \overline{\mathrm{Pic}}_{X/S}^d$. The fiber of $A^d$ over a point $[I] \in \overline{\mathrm{Pic}}^d(X_s)$ is the projective space $\mathbb{P}(H^0(X_s, I)^\vee)$.


\begin{theorem}[Theorem \ref{thm:SOD:curve}]
\label{thm:intro:SOD:curve}
There are natural derived enhancements of the Abel maps,
\[
\mathbf{A}^d \colon \mathbf{Quot}_{\omega/X/S}^d \to \overline{\mathrm{Pic}}_{X/S}^d, \qquad
\mathbf{A}_\omega^d \colon \mathbf{Quot}_{\omega/X/S}^{2p_a-2-d} \to \overline{\mathrm{Pic}}_{X/S}^{d},
\]
which are quasi-smooth (i.e., locally complete intersection in the derived sense) for all $d \in \mathbb{Z}$. Furthermore, for every integer $d \geq p_a - 1$, there is a semiorthogonal decomposition
\[
\mathrm{D}(\mathbf{Quot}_{\omega/X/S}^d) = \left\langle \mathrm{D}(\mathbf{Quot}_{\omega/X/S}^{2p_a-2-d}),\ \text{$(1-p_a+d)$ copies of } \mathrm{D}(\overline{\mathrm{Pic}}_{X/S}^d) \right\rangle,
\]
where $\mathrm{D}$ denotes any of the standard derived categories: $\Dqc$, $\mathrm{D}^-_{\mathrm{coh}}$, $\mathrm{D}^\mathrm{b}_{\mathrm{coh}}$, or $\mathrm{Perf}$.
\end{theorem}

This theorem provides a categorification of the fiberwise Riemann--Roch formula:
\[
\dim_{\kappa(s)} H^0(X_s, I) - \dim_{\kappa(s)} H^1(X_s, I) = 1 - p_a + d,
\]
where $I \in \overline{\mathrm{Pic}}^d(X_s)$ is a rank-one, torsion-free sheaf of degree $d$ on the curve $X_s$.

Our strategy in \S\ref{subsec:Abel:cPic} also extends to suitable derived enhancements of $\overline{\mathrm{Pic}}^d(X_s)$.

It is worth noting that generally the compactified Picard schemes $\overline{\mathrm{Pic}}^d(X_s)$ can be highly singular, even when endowed with their natural derived structures. Nevertheless, the derived Abel maps constructed here are always quasi-smooth, and the theorem applies even when neither the source nor the target of the Abel map is quasi-smooth.

Several notable cases of the theorem include:
\begin{enumerate}
    \item \emph{Gorenstein case  (\S \ref{subsec:Abel:Gorenstein}):} If all the geometric fibers $X_s$ of $X/S$ are Gorenstein,  the relative dualizing sheaf $\omega$ is invertible, and tensoring with $\omega$ gives a canonical isomorphism
    \[
    \Hilb_{X/S}^d \xrightarrow{\sim} \Quot_{\omega/X/S}^d.
    \]
For clarity, we switch the upper and lower indices from the Quot scheme case and define $A_d = A^d \circ (\underline{\hphantom{A}} \otimes \omega)$ and $A^\omega_d = A^d_\omega \circ (\underline{\hphantom{A}} \otimes \omega)$. The Abel maps then become
\begin{align*}
A_d   \colon \Hilb_{X/S}^d \to \overline{\mathrm{Pic}}_{X/S}^d, \qquad
A^\omega_d \colon \Hilb_{X/S}^{2p_a-2-d} \to \overline{\mathrm{Pic}}_{X/S}^{d},
\end{align*}
where a zero-dimensional subscheme $Z \subseteq X_s$ of length $d$ (respectively, length $2p_a-2-d$) is sent to the sheaf $\sHom(I_Z, \mathscr{O}_{X_s})$ (respectively, $I_Z \otimes \omega$), with $I_Z$ the ideal sheaf of $Z$. 

 The fiber of the Abel map $A_d$ over a point $J \in \overline{\mathrm{Pic}}^d(X_s)$ is the projective space $\mathbb{P}(H^0(X_s, J)^\vee)$, which coincides with the complete linear system $|J|$ of {\em generalized divisors} of the rank-$1$, torsion-free sheaf $J$, studied by Hartshorne \cite{Har86, Har94, Har07}.

    In this case, our theorem provides natural derived enhancements $\mathbf{Hilb}_{X/S}^d$ of the classical Hilbert schemes $\Hilb_{X/S}^d$ for all $d$, and tensoring with $\omega$ induces an isomorphism of derived stacks $\mathbf{Hilb}_{X/S}^d \xrightarrow{\sim} \mathbf{Quot}_{\omega/X/S}^d$. \footnote{Note that we define $\mathbf{Hilb}_{X/S}^d$ to be $\mathbf{Quot}_{\mathscr{O}_X/X/S}^d$, in line with the classical construction of \cite{AK80}. This differs from the approach in \cite{CK01, CK02}, where $\mathbb{R}\mathrm{Hilb}$ and $\mathbb{R}\mathrm{Quot}$ generally have distinct derived structures.} 
    The semiorthogonal decomposition of Theorem~\ref{thm:intro:SOD:curve} holds verbatim after replacing all $\mathbf{Quot}^d_{\omega/X/S}$ by $\mathbf{Hilb}^d_{X/S}$ (see Corollary~\ref{cor:Abel:Gorenstein}). Notably, while the classical Hilbert schemes can be highly singular, the derived schemes $\mathbf{Hilb}_{X/S}^d$ are always quasi-smooth over $\overline{\mathrm{Pic}}_{X/S}^d$.

    \item \emph{Locally planar case (\S \ref{subsec:Abel:planar}):} If $X/S$ is a family of locally planar curves, then the derived Hilbert schemes are classical (Lemma \ref{lem:planar:classical}):
    \[
    \mathbf{Hilb}_{X/S}^d = \Hilb_{X/S}^d,
    \]
    and the theorem gives a semiorthogonal decomposition for all $d \geq p_a - 1$:
    \[
    \mathrm{D}(\Hilb_{X/S}^d) = \left\langle \mathrm{D}(\Hilb_{X/S}^{2p_a-2-d}),\ \text{$(d-p_a+1)$ copies of } \mathrm{D}(\overline{\mathrm{Pic}}_{X/S}^d) \right\rangle.
    \]
    This greatly extends the semiorthogonal decompositions for symmetric powers of smooth complex projective curves due to Toda~\cite[Corollary~5.11]{Tod2} (see also \cite[Corollary~3.10]{JL18}, \cite{BK19}, and \cite[Example~3.33]{J22a}). 
    \end{enumerate}

   As discussed in \cite{J23}, the above semiorthogonal decomposition can be regarded as a categorification of the Beilinson--Bernstein--Deligne--Gabber decomposition theorem for the map $A^d$ in this context.    
    In this sense, our result provides a coherent counterpart to the generalized Macdonald formula of Maulik--Yun \cite{MY} and Migliorini--Shende \cite{MS} (see also \cite{Ren}).
    
    Notably, in contrast to previous approaches in the constructible setting \cite{MY, MS, MT}, which typically rely on the smoothing techniques of \cite{She12}, our use of derived techniques enables us to work directly with singular families $X/S$ without smoothing.
    
When $X/S$ is a Gorenstein integral curve $C/\mathbb{C}$ lying on a Calabi--Yau 3-fold, Theorem~\ref{thm:intro:SOD:curve} can be viewed as a categorification of the wall-crossing formula of Pandharipande and Thomas~\cite{PT}:
\begin{align*}
    P_{m,C} - P_{-m,C} = (-1)^{m-1} m N_{1,C},
\end{align*}
where $P_{\pm m,C}$ and $N_{1,C}$ denote the local stable pairs and local DT invariants, respectively. A categorification of this formula in the context of categorical DT theory has been obtained by Toda~\cite[Theorem 5.1.9, Example 5.1.11]{Tod24}.

We also note that families of integral curves and their compactified Picard schemes arise naturally in various contexts, including the study of moduli spaces of torsion sheaves on K3 surfaces \cite{ADM} and  in Hitchin fibrations for Higgs bundles \cite{MS24, MSY25, HMMS, PT24}.

Finally, we expect the results of the theorem to be optimal: for a smooth projective curve $C/\mathbb{C}$, Lin~\cite{Lin} shows that $\mathrm{D}(\operatorname{Sym}^d(C))$ is indecomposable for $d \leq g(C) - 1$. Similarly, we anticipate that $\mathrm{D}(\mathbf{Quot}_{\omega/X/S}^d)$ is indecomposable over $S$ for $d \leq p_a- 1$.

\subsection{Abel maps, torsion-free sheaves and coherent systems}
To prove the Theorem \ref{thm:intro:SOD:curve}, we introduce a general framework for Abel maps associated to families of integral curves $X/S$ (\S\ref{subsec:Abel:general}). Specifically, for any suitable derived moduli stack $\mathscr{M}$, we define Abel maps and establish corresponding semiorthogonal decompositions (Theorems~\ref{thm:SOD:abel:LinSys}\& \ref{thm:SOD:abel:GLinSys}).

This generality allows us to treat not only compactified Picard schemes, but also higher rank torsion-free sheaves, and \emph{coherent systems} (i.e., Brill--Noether pairs in the sense of~\cite{KN95}). In particular, our framework encompasses generalizations of the classical varieties of linear series $G_d^\ell$ studied in ~\cite[Chapter~IV]{ACGH} (see Corollary~\ref{cor:SOD:Gdl}), as well as extensions of the moduli of Thaddeus pairs~\cite{Tha} from vector bundles to torsion-free sheaves.

For example, when $\mathscr{M} = \underline{\mathrm{Coh}}_{X/S}^{\mathrm{tf}}(\rho, d)$, the moduli stack of torsion-free sheaves of rank $\rho$ and degree $d$ on $X/S$, we obtain derived enhancements of the moduli spaces $G(\rho, d, \ell)$ of $\ell$-dimensional coherent systems considered in \cite{KN95}. The Abel maps for these systems are
\begin{align*}
\mathbf{A}^{\rho, d, \ell} \colon  \mathbf{G}_{X/S}(\rho, d, \ell) &\to  \underline{\mathrm{Coh}}_{X/S}^{\mathrm{tf}}(\rho, d), \\
\mathbf{A}^{\rho, d, \ell}_\omega \colon  \mathbf{G}^\omega_{X/S}(\rho, d, \ell) &\to  \underline{\mathrm{Coh}}_{X/S}^{\mathrm{tf}}(\rho, d),
\end{align*}
where, for a morphism $s \colon \Spec\kappa(s) \to S$, the $s$-points of $  \mathbf{G}_{X/S}(\rho, d, \ell)$ and $\mathbf{G}^\omega_{X/S}(\rho, d, \ell)$  parametrize coherent systems $(E, V)$ and $(E, W)$: here $E$ is a torsion-free sheaf of rank $\rho$ and degree $d$ on $X_s$, and $V \subseteq H^0(X_s, E)$, $W \subseteq H^1(X_s, E)^\vee$ are linear subspaces of dimension~$\ell$. The Abel maps simply record the underlying torsion-free sheaf $E$ of each coherent system.

\begin{theorem}[Corollary~\ref{cor:SOD:coh.sys}]
\label{thm:intro:SOD:coh.sys}
The Abel maps $\mathbf{A}^{\rho, d, \ell}$ and $\mathbf{A}^{\rho, d, \ell}_\omega$ are quasi-smooth. If $\ell = 1$ or if $S$ is a $\mathbb{Q}$-scheme, then for all $d \geq (p_a - 1)\rho$, there is an $\mathscr{M}$-linear semiorthogonal decomposition:
\[
\D\left(\mathbf{G}_{X/S}(\rho, d, \ell)\right) =
\left\langle
    \text{$\binom{(1-p_a)\rho+d}{j}$ copies of }
    \D\left(\mathbf{G}^\omega_{X/S}(\rho, d, \ell-j)\right)
    \,\Big|\, 0 \leq j \leq \ell
\right\rangle.
\]
\end{theorem}
The $\mathscr{M}$-linearity ensures that the semiorthogonal decomposition is preserved under base change along any morphism to $\mathscr{M}$.  This, in particular, allows one to restrict the decomposition to open substacks such as $\mathcal{B}un_{X/S}(\rho, d)$, the moduli stack of vector bundles, as well as to any other substacks of interest.

Our results considerably extend the range of known semiorthogonal decompositions for moduli spaces associated with curves. In particular, this framework generalizes earlier results for symmetric powers of smooth curves~\cite{Tod2, JL18, BK19, J22a, KT} and incorporates recent developments for varieties of linear series on general curves~\cite{J23, Tod23}. In the case where $\ell = 1$, the theorem above also extends a result of Koseki and Toda \cite{KT} on semiorthogonal decomposition for moduli of Thaddeus pairs for stable vector bundles on smooth curves to the context of torsion-free sheaves on integral curves.

We also emphasize that, although the main focus of this paper is on the case of curves, the general methodology of this paper applies more broadly, including to stable pair moduli spaces, surfaces, two-dimensional categories, and Calabi--Yau families; see \S\ref{subsec:SOD:Ext} for further details.

\subsection{The $\mathcal{Q}$-complex and moduli of extensions}

We work within the framework of derived algebraic geometry, as developed by Lurie, To{\"e}n, Vezzosi, and others~\cite{DAG, HTT, SAG, ToenDAG, GRI}. A central role in our approach is played by the theory of derived projectivizations and Grassmannians, introduced by the author in~\cite{J22a, J22b, J23}.

Our methodology is based on two key technical ingredients:
\begin{enumerate}
    \item The $\mathcal{Q}$-complex, developed in Section~\ref{sec:Q-complex}, which is a derived generalization of  Grothendieck's $Q$-module~\cite[III$_2$~\S7.7]{EGA} and the Altman--Kleiman $H$-module~\cite[(1.1)]{AK80}.
    \item A derived theory of moduli of extensions, presented in Section~\ref{sec:extension}, extending Lange's classical theory~\cite{Lag83} to the derived context.
\end{enumerate}

Specifically, in Section~\ref{sec:Q-complex}, for a morphism $f \colon \mathscr{X} \to \mathscr{M}$ of derived stacks and complexes $\mathscr{E}, \mathscr{F} \in \Dqc(\mathscr{X})$, we construct, under mild assumptions, a complex $\shQ_f(\mathscr{E}, \mathscr{F})$ that serves as a dual to the pushforward $f_* \sHom(\mathscr{E}, \mathscr{F})$, without assuming perfectness of $\mathscr{F}$ or finite Tor-amplitude for $f$. The formation of $\shQ_f(\mathscr{E}, \mathscr{F})$ is compatible with base change, and we provide fiberwise criteria for its Tor-amplitude (Corollary~\ref{cor:Ext:fiberwise}).

The $\mathcal{Q}$-complex is a versatile tool with broad applicability and independent interest. For example, it yields an explicit description of the cotangent complex of the derived moduli stack of coherent sheaves $\underline{\mathbf{Coh}}(X)$ on a proper scheme $X$ (Lemma~\ref{lem:cotangent:Coh}), complementing the foundational existence result of Porta--Sala~\cite[Lemma 2.21]{PS}.

In Section~\ref{sec:extension}, we use the $\mathcal{Q}$-complex, together with derived projectivizations and Grassmannians, to develop a natural framework for moduli stacks of extensions. Specifically, given $f \colon \mathscr{X} \to \mathscr{M}$ and $\mathscr{E}, \mathscr{F} \in \Dqc(\mathscr{X})$ as above, we construct derived moduli stacks $\mathbf{PExt}_f^i(\mathscr{E}, \mathscr{F})$ parametrizing non-split $i$th extensions of $\mathscr{E}$ by $\mathscr{F}$ over $\mathscr{M}$ (\S \ref{subsec:PExt}), as well as Grassmannian-type moduli stacks of extensions $\mathbf{GExt}_f^i(\mathscr{E}, \mathscr{F}; \ell)$ (\S \ref{subsec:GExt}).

These constructions are especially useful in Brill--Noether contexts (see \S\ref{sec:BN}), extending the Hecke moduli spaces studied in \cite{Neg19, Neg22, J22b}. 

In \S\ref{subsec:SOD:Ext}, we further establish semiorthogonal decompositions for these moduli stacks of extensions (Theorems~\ref{thm:SOD:PExt} and~\ref{thm:SOD:GExt}).

\subsection{Related works}
The study of compactified Picard schemes and Abel maps has a long and rich history; see \cite{AIK77, AK79, AK80, D'so79, Reg80, Kas13} and references therein for background. Our approach in the derived setting parallels many of the foundational arguments of Altman and Kleiman~\cite{AK79, AK80}, particularly their analysis of the module $H(I,F)$ and the study of moduli of linear systems.

Coherent systems (Brill--Noether pairs) for vector bundles and torsion-free sheaves on curves have also been widely studied; see, for instance, \cite{BG91, Ber94, LePot93, Tha, KN95, Bra09, BGMN03, Bho09} for further references.

Our derived approach builds on the theory of projectivizations and Grassmannians of complexes, developed extensively by the author in \cite{JL18, J20, J21, J22a, J22b, J19, J23}. The current work significantly broadens the applicability of these results.

In their seminal works \cite{CK01, CK02}, Ciocan-Fontanine and Kapranov introduced derived Quot and Hilbert schemes. Our approach utilizing derived projectivizations differs from theirs. Nevertheless, we expect that for suitable derived structures on the compactified Picard schemes, our derived Quot scheme $\mathbf{Quot}_{\omega/X/S}^d$ matches their $\mathbb{R}\mathrm{Quot}$.

As discussed, Theorem~\ref{thm:intro:SOD:curve} gives a coherent analogue of the generalized Macdonald formula due to Maulik--Yun~\cite{MY}, Migliorini--Shende~\cite{MS}; see also Rennemo \cite{Ren}. It also yields a categorification of the wall-crossing formula of Pandharipande and Thomas~\cite{PT}, paralleling the approach in categorical DT theory developed by Toda in \cite[Theorem 5.1.9]{Tod24}.

Families of integral curves and their compactified Picard schemes naturally arise in various settings, such as in the study of moduli spaces of torsion sheaves on K3 surfaces \cite{ADM}, and in Hitchin fibrations for Higgs bundles over the elliptic locus \cite{MY, Ari, MS24, MSY25, HMMS, PT24}. We hope to explore related applications in future work.

As mentioned, our results on semiorthogonal decompositions for coherent systems (Theorem~\ref{thm:intro:SOD:coh.sys}) extend previous work for smooth curves: see \cite{Tod2, JL18, BK19, J22a, KT} for symmetric powers, \cite{J21, J23, Tod23} for linear series, and \cite{KT} for Thaddeus pairs.

Our constructions in Section~\ref{sec:extension} extend the classical framework of \cite{Lag83} to the derived setting and are closely linked to geometric representation theory. Porta and Sala \cite{PS} have studied derived moduli of extensions, which are fundamental in their theory of categorified Hall algebras; see also \cite{DPS22, DPSV23}. Derived Hecke moduli spaces (see Example~\ref{eg:Hecke}) were studied in \cite{Neg19, Neg22, J22b}, and have been used in \cite{MMSV} to explore the COHA of zero-dimensional sheaves on surfaces.
The geometry of moduli of extensions is closely linked to Brill--Noether theory (see \S\ref{sec:BN}) \cite{BCJ, Bay15, Z24b} and to wall-crossing phenomena \cite{Kos22, Z24a}. Our constructions in \S\ref{subsec:SOD:Ext} are also related to the study of stable pair moduli spaces (see Example~\ref{eg:stable.pairs}), as developed by Pandharipande--Thomas \cite{PT09, PT}, Bridgeland \cite{Bri11}, and Bryan--Steinberg \cite{BS16, DPSV23}, as well as to Brill--Noether settings for two-dimensional categories, as studied in \cite{Che21} and \cite[\S 5.3]{J19}.

 Recent work by the author with Li and Zhao~\cite{JLZ} reveals profound representation-theoretic structures underlying the semiorthogonal decompositions of derived Grassmannians, with notable applications to instanton moduli on blowups of surfaces studied in \cite{LQ99, NY}.  Combined with the general framework developed in this paper, these developments provide a robust approach to a wide range of moduli problems, which we plan to pursue in future work.
\subsection{Notations and conventions}

We adopt the following notational and terminological conventions throughout this paper, aiming to make our exposition accessible to readers familiar with classical algebraic geometry as well as those versed in derived settings.

\begin{itemize}

\item {\bf Derived conventions for functors.} 
All functors are interpreted in the \emph{derived} sense. For example, for a morphism $f \colon X \to Y$ of derived stacks, we write the pullback as $f^* \colon \Dqc(Y) \to \Dqc(X)$ and the tensor product in $\Dqc(X)$ as $\otimes$, where $\Dqc(X)$ denotes the stable $\infty$-category $\mathrm{QCoh}(X)$ defined in \cite[\S 6.2.2]{SAG}.

When $f \colon X \to Y$ is a morphism of underived schemes, the homotopy category of $\Dqc(X)$ recovers the usual triangulated quasi-coherent derived category $D_{QCoh}(\mathscr{O}_X)$ of \cite[\href{https://stacks.math.columbia.edu/tag/06YZ}{Tag 06YZ}]{stacks-project}. In this case, our notation $f^*$ corresponds to the classical derived pullback $\mathbb{L}f^*$, $f_*$ to the derived pushforward $\mathbb{R}f_*$, $\otimes$ to the derived tensor product $\otimes^\mathbb{L}$, and $\sHom_X(-,-)$ to the derived sheaf Hom $\mathbb{R}\sHom_X(-,-)$, whenever defined.

The only exception is that we use $\operatorname{Hom}(A, B) = \Ext^0(A,B)$ for the classical Hom group in abelian categories, and $\mathbb{R}\operatorname{Hom}(A, B)$ for the total derived Hom group.

\item {\bf Classical terminology for sheaves and complexes.} 
To avoid confusion for classical readers, we refer to objects of $\Dqc(X)$ as \emph{quasi-coherent complexes}, even when $X$ is a derived stack. Note, however, that such objects are not necessarily represented by actual complexes. We reserve the terminology ``sheaves" for discrete objects in $\Dqc(X)$, that is, those whose homotopy groups (or, classically, sheaf cohomology) vanish in all nonzero degrees. This aligns with classical terminology for sheaves and complexes on schemes, but differs from the conventions in Lurie~\cite{SAG}.

\item {\bf Grading conventions.} 
We adopt homological grading conventions, as in Lurie~\cite{DAG, HA, SAG}, rather than the cohomological grading of~\cite{SGA6} and the Stacks project~\cite{stacks-project}. For instance, for a classical ring $R$, an $R$-module $M$ has Tor-amplitude in $[a, b]$ in our convention if and only if it has Tor-amplitude in $[-b, -a]$ in the conventions of~\cite[Expos\'e III]{SGA6} and~\cite[\href{https://stacks.math.columbia.edu/tag/0651}{Tag 0651}]{stacks-project}.

\item {\bf Grothendieck’s convention.} 
For projectivizations and Grassmannians, we will use Grothendieck’s quotient convention as in~\cite{J22a, J22b}. For example, for a $\kappa$-vector space $V$, the projectivization $\mathbb{P}(V)$ parametrizes one-dimensional quotients of $V$, so that $\mathbb{P}(V)(\kappa) \simeq (V^\vee \setminus \{0\})/\kappa^\times$, where $V^\vee = \operatorname{Hom}_\kappa(V, \kappa)$ denotes the dual vector space.

\item {\bf Derived stacks.} 
We use boldface letters such as $\mathbf{Hilb}_{X/S}^d$ and $\mathbf{Quot}_{\omega/X/S}^d$ for derived enhancements of classical moduli stacks. The corresponding classical stacks are denoted with standard roman letters, e.g., $\Hilb_{X/S}^d$ and $\Quot_{\omega/X/S}^d$.

\end{itemize}
In Sections~\ref{sec:Q-complex} and~\ref{sec:extension}, where we develop the necessary theories within the framework of derived algebraic geometry, we will use terminology that is largely standard and mainly follows Lurie~\cite{DAG, HA, SAG}, subject to the conventions described above.

Let $\CAlgDelta$ denotes the $\infty$-category of simplicial commutative rings (that is, animated rings) as defined in \cite[\S 25.1]{SAG}. A {\em prestack} is a functor $X$ from $\CAlgDelta$ to the $\infty$-category $\mathcal{S}$ of spaces (i.e., $\infty$-groupoids). A morphism $f \colon X \to Y$ between prestacks is defined as a natural transformation of functors. The morphism $f$ is called representable if for each $\Spec A \to Y$, the base change $X_A = X \times_Y \Spec A$ is a derived scheme. 

A \emph{derived stack} $X$ is a prestack satisfying {\'e}tale descent. Let $X_{\mathrm{cl}} \colon \CAlgDelta^{\heartsuit} \to \mathcal{S}$ denote the sheafification of the restriction $X|_{ \CAlgDelta^{\heartsuit}}$ to discrete commutative rings; we call $X_{\mathrm{cl}}$ the \emph{classical truncation} of $X$. If $X = (|X|, \mathscr{O}_X)$ is a derived scheme, then $X_{\mathrm{cl}} = (|X|, \pi_0 \mathscr{O}_X)$.

A {\em point} of a derived stack $X$ is an equivalence class of morphisms $x \colon \operatorname{Spec} \kappa \to X$, where $\kappa$ is a field. By abuse of notation, we write $x \in X$ for such a point, and choose a representative $x \colon \operatorname{Spec} \kappa(x) \to X$. If $X$ is a derived scheme, points correspond to those of the underlying topological space $|X|$, and $\kappa(x)$ may be taken as the residue field at $x \in |X|$.

A morphism $X \to Y$ between derived stacks is called \emph{representable} if, for any $\operatorname{Spec} A \to Y$, where $A \in \CAlgDelta$, the fiber product $X_{A} : =X \times_{Y} \operatorname{Spec} A$ is a derived scheme. 

For a derived stack $X$, we consider the following derived $\infty$-categories:
\[
\Perf(X) \subseteq \mathrm{D}^-_{\mathrm{coh}}(X) \subseteq \mathrm{D}^b_{\mathrm{coh}}(X) \subseteq \Dqc(X).
\]
Here, $\Dqc(X)$ denotes the stable $\infty$-category of quasi-coherent complexes on $X$ as defined in \cite[\S 6.2.2]{SAG}, and $\Perf(X)$ is its full subcategory of perfect complexes. Both categories admit explicit descriptions as limits:
\[
\Dqc(X) \simeq \varprojlim_{\Spec A \to X} \Dqc(\Spec A), \qquad
\Perf(X) \simeq \varprojlim_{\Spec A \to X} \Perf(\Spec A),
\]
where $\Dqc(\Spec A)$ is the $\infty$-category $\mathrm{Mod}_{A^\circ}$ of module spectra over the underlying $\mathbb{E}_\infty$-ring $A^\circ$ of $A$ (\cite[Notation 7.1.1.1]{HA}), and $\Perf(\Spec A)$ is the stable subcategory of perfect objects (\cite[Definition 7.2.4.1]{HA}). The subcategories $\mathrm{D}^-_{\mathrm{coh}}(X)$ and $\mathrm{D}^b_{\mathrm{coh}}(X)$ consist of almost perfect complexes (see \cite[Proposition 6.2.5.2(6)]{SAG}) and locally truncated (see \cite[Proposition 3.5.3]{DAG}, \cite[Notation 6.4.1.1]{SAG}) almost perfect complexes, respectively.

For a derived stack $X$, $\Dqc(X)$ admits a natural $t$-structure $(\Dqc(X)_{\ge 0}, \Dqc(X)_{\le 0})$, where $\Dqc(X)_{\ge 0} = \Dqc(X)^{\mathrm{cn}}$ is the subcategory of connective complexes.
For an integer $n$, we let $\Dqc(X)_{\ge n} = \Dqc(X)_{\ge 0}[n]$ and $\Dqc(X)_{\le n} = \Dqc(X)_{\le 0}[n]$, where $[n] = \Sigma^{\circ n}$ is the $n$th suspension functor. We let $\tau_{\ge n} \colon \Dqc(X) \to \Dqc(X)_{\ge n}$ (resp. $\tau_{\le n} \colon \Dqc(X) \to \Dqc(X)_{\le n}$) denote the right (resp. left) adjoint functor of the inclusion functor. We let $\Dqc(X)^{\heartsuit} = \Dqc(X)_{\ge 0} \cap \Dqc(X)_{\le 0}$ and refer to its objects as {\em (discrete) sheaves}. We let 
	$$\pi_{n} = \tau_{\ge 0} \circ \tau_{\le 0} \circ [-n] \simeq \tau_{\le 0} \circ \tau_{\ge 0} \circ [-n] \colon \Dqc(X) \to \Dqc(X)^{\heartsuit}.$$


A morphism $f \colon X \to Y$ of derived stacks is called \emph{smooth} if it is locally almost of finite presentation and its relative cotangent complex $\mathbb{L}_{X/Y}$ is perfect of Tor-amplitude $\le 0$.

Following \cite[\S 1.3.3]{TVe08}, we define by induction the notation of \emph{$n$-geometric} derived stacks for $n \ge -1$. The $(-1)$-geometric derived stacks (and morphisms) are precisely the derived affine schemes (and affine morphisms). Assuming the notions of $(n-1)$-geometric stacks and morphisms have been defined, a derived stack $X$ is $n$-geometric if its diagonal is $(n-1)$-geometric and there exists a disjoint union of derived affine schemes $U$ together with a smooth $(n-1)$-geometric epimorphism $U \to X$. A morphism $f \colon X \to Y$ of derived stacks is $n$-geometric if for any morphism $\operatorname{Spec} A \to Y$ with $A \in \CAlgDelta$, the base change $X \times_{Y} \operatorname{Spec} A$ is $n$-geometric. 

A derived stack $X$ is called \emph{geometric} if it is $n$-geometric for some $n$. This definition agrees with those in \cite{ToenDAG, GRI, PS} and is equivalent to the notion of derived Artin stack in \cite{Khan}. By \cite[Proposition 2.2.5.1]{TVe08}, a morphism $f \colon X \to Y$ of derived geometric stacks is smooth if and only if it is a geometric smooth morphism in the sense of \cite[\S 2.2.3]{TVe08}.

A derived stack $X$ is said to be a {\em perfect stack} (\cite[Definition 3.2]{BFNjams}) 
  if it has affine diagonal and the inclusion $\Perf(X) \subseteq \Dqc(X)$ induces an equivalence 
	\[\Ind \Perf(X) \simeq \Dqc(X).\] 
A morphism of derived stacks $X \to Y$ is said to be \emph{perfect} if for all morphisms $U= \operatorname{Spec} (A) \to Y$ from an affine derived scheme $U$, the fibers $X \times_U Y$ are perfect. 

For a perfect stack $X$, the stable $\infty$-category $\Dqc(X)$ is compactly generated, with compact objects given by the perfect complexes. Moreover, the natural $t$-structure on $\Dqc(X)$ is accessible, both left and right complete, and compatible with filtered colimits; see \cite[Corollary 9.1.3.2(2), (3), (4)]{SAG}.

Given a morphism $f \colon X \to Y$ of derived stacks, we write $f^* \colon \Dqc(Y) \to \Dqc(X)$ for the pullback functor. If $f$ is perfect, $f^*$ admits a right adjoint, the pushforward functor $f_* \colon \Dqc(X) \to \Dqc(Y)$, which preserves small colimits. The projection formula and base-change formula hold for perfect morphisms; see \cite[Proposition 9.1.5.7]{SAG}.

For the reader’s convenience, we note that most morphisms considered in this paper are perfect. This includes, for instance, all morphisms between perfect stacks \cite[Corollary 3.23]{BFNjams}, as well as morphisms that are relative quasi-compact, separated derived algebraic spaces (see \cite[Proposition 9.6.1.1]{SAG} or \cite[Proposition 3.19]{BFNjams}). A special case of the latter case is that of locally (quasi-)projective morphisms.


For a perfect stack $X$, since $\Dqc(X)$ is presentable and the tensor product $\otimes \colon \Dqc(X) \times \Dqc(X) \to \Dqc(X)$ commutes with small colimits separately in each variable, one obtains that for every pair $\sE, \sF \in \Dqc(X)$, there exists a quasi-coherent complex $\sHom_X(\sE,\sF) \in \Dqc(X)$, called {\em (internal) $\Hom$-complex}, which satisfies the following universal property: there is an evaluation map $\sE \otimes \sHom_X(\sE,\sF) \to \sF$ such that, for every $\sG \in \Dqc(X)$,
the composite map
	\begin{align*}
		\Map_X(\sG, \sHom_X(\sE, \sF)) \to \Map_X(\sE \otimes \sG, \sE \otimes \sHom_X(\sE, \sF)) \to \Map_X(\sE  \otimes \sG, \sF)	
	\end{align*}
is a homotopy equivalence; see \cite[\S 9.5.3]{SAG} for more details.

\subsection{Acknowledgement}
The author thanks Arend Bayer, Richard Thomas, Yukinobu Toda, Tom Bridgeland, Junliang Shen, Qizheng Yin, Tudor Paradariu, Yu Zhao, and Francesco Sala for  insightful discussions related to this work. Thanks are also due to Adeel Khan for many helpful conversations on derived stacks. Special thanks go to Mauro Porta for his careful reading of an early draft and for numerous valuable comments and suggestions.
This work was supported by the Hong Kong Research Grants Council ECS grant (Grant No. 26311724).

\section{The $\mathcal{Q}$-complex}
\label{sec:Q-complex}
Let $f \colon X \to S$ be a map of derived stacks, and let $\mathscr{E}, \mathscr{F} \in \Dqc(X)$. Under suitable assumptions, we construct a complex $\shQ_{f}(\sE,\sF)$ (Construction \ref{constr:Q}) that serves as a dual to the pushforward of Hom complex $f_* \sHom_X(\mathscr{E}, \mathscr{F})$. 
This construction generalizes the classical Grothendieck $Q$-module \cite[III$_2$ Theorem 7.7.6]{EGA} and the Altman--Kleiman's module $H(\mathscr{E}, \mathscr{F})$ \cite[(1.1)]{AK80} to the derived setting (see Remark \ref{rmk:H(I,F)}). 

 The formation of $\shQ_{f}(\sE,\sF)$ is compatible with base change in $S$ (Proposition \ref{thm:Q}). We establish fiberwise criteria for the connectivity and Tor-amplitude of $\shQ_{f}(\sE, \sF)$ (Corollary~\ref{cor:Ext:fiberwise}), and formulate versions of relative Serre duality and coherent duality in terms of the complex $\mathcal{Q}$ (Propositions~\ref{prop:Serre:Q} \& \ref{prop:coh:duality}). Furthermore, the $\mathcal{Q}$-complex offers a derived framework for simplicity for complexes (\S\ref{subsec:simple}), and provides an explicit description of the cotangent complexes of derived moduli stacks of coherent sheaves $\underline{\mathbf{Coh}}(X)$ on proper complex schemes (Lemma~\ref{lem:cotangent:Coh}).

A key notion in this context is that of relative perfectness, which replaces the absolute perfectness assumption on $\mathscr{F}$; see \cite[Exposé III]{SGA6} for the classical theory. 

\subsection{Relative Tor-amplitudes}
We now consider relative Tor-amplitude and relative perfectness in the derived setting, generalizing the classical notions of \cite[Exposé III, \S 3]{SGA6}.

\begin{definition}[Relative Tor-amplitude and relative perfectness; see also {\cite[\S 2.1]{PS}}] 
\label{def:relative.Tor}
Let $f \colon X \to S$ be a morphism of derived stacks, and let $\sF \in \Dqc(X)$. Let $a \le b$ and $n$ be integers.
\begin{enumerate}
    \item We say that $\sF$ has \emph{relative Tor-amplitude in $[a,b]$} (respectively, \emph{relative Tor-amplitude $\le n$}) over $S$ if, for every discrete sheaf $\sG \in \Dqc(S)^\heartsuit$, 
    \[
        \pi_i(\sF \otimes f^*(\sG)) = 0 \quad \text{for } i \not\in [a,b] \quad (\text{respectively, for } i > n).
    \]
    We say that $\sF$ has \emph{relative finite Tor-amplitude over $S$} if there exist integers $a \le b$ such that $\sF$ has relative Tor-amplitude in $[a,b]$ over $S$. We say that $\sF$ has \emph{locally finite relative Tor-amplitude} over $S$ if for every morphism $\operatorname{Spec} A \to S$ with $A \in \CAlgDelta$, the base change of $\sF$ to $X \times_S \operatorname{Spec} A$ has finite Tor-amplitude over $\operatorname{Spec} A$.
    \item Suppose further that $f$ is locally almost of finite presentation. We say that $\sF$ has \emph{relative perfect amplitude in $[a,b]$} (respectively, $\le n$) over $S$ if $\sF$ is almost perfect and has relative Tor-amplitude in $[a,b]$ (respectively, $\le n$) over $S$. We say that $\sF$ is \emph{relatively perfect over $S$} if it is almost perfect and has locally finite relative Tor-amplitude over $S$.
\end{enumerate}
\end{definition}

When $f = \operatorname{id}_X$, these definitions reduce to the usual notions of Tor-amplitude and perfectness on $X$. Note also that $\sF$ has relative Tor-amplitude in $[a,b]$ if and only if $\sF$ has relative Tor-amplitude $\le b$ and $\pi_i(\sF) = 0$ for $i < a$.

Let $\mathrm{Perf}(X/S)$ denote the full subcategory of $\Dqc(X)$ consisting of complexes that are relatively perfect over $S$. 
If $X \to \operatorname{Spec} k$ is a Noetherian scheme over a field, then $\mathrm{Perf}(X/k) = \mathrm{D^b_{coh}}(X)$, the category bounded coherent complexes.
If $f =  \operatorname{id}_X$, then $\mathrm{Perf}(X/X) = \mathrm{Perf}(X)$ is the usual category of perfect complexes on $X$.

\begin{remark}
\label{rmk:relative-Tor_preserves_truncated}
If $S$ is a perfect stack, the condition that $\sF$ has relative Tor-amplitude $\le n$ over $S$ is equivalent to the following apparently stronger statement:
\begin{itemize}
    \item The functor $\sG \mapsto \sF \otimes f^*(\sG)$ restricts to
    \[
        \sF \otimes f^*(\blank) \colon \Dqc(S)_{\le 0} \to \Dqc(X)_{\le n}.
    \]
\end{itemize}
In fact, consider any $\sG \in \Dqc(S)_{\le 0}$. Since the $t$-structure on $\Dqc(S)$ is right complete (\cite[Corollary 9.1.3.2.(4)]{SAG}), we can write $\sG = \varinjlim \tau_{\ge -m}(\sG)$. Using the fiber sequence
\[
    \tau_{\ge 1-m}(\sG) \to \tau_{\ge -m}(\sG) \to \pi_{-m} \sG[-m],
\]
and induction on $m \ge 0$, 
we obtain that $\pi_i(\sF \otimes f^*(\sG)) = 0$ for all $i > n$.
\end{remark}

The next lemma is a derived analogue of \cite[{E}xpos{\'e} III, Corollaries 3.7.1 and 4.8.1]{SGA6}:

\begin{lemma}
\label{lem:relative-perf}
Let $f \colon X \to S$ be a representable morphism of derived stacks, locally almost of finite presentation, and let $\sF \in \Dqc(X)$.
\begin{enumerate}
    \item \label{lem:relative-perf-1}
    If $\sF$ has finite relative Tor-amplitude over $S$, then $f_*(\sF)$ has finite (absolute) Tor-amplitude.
    \item \label{lem:relative-perf-2}
    If $f$ is proper and $\sF$ is relatively perfect over $S$, then $f_*(\sF)$ is perfect.
\end{enumerate}
\end{lemma}

\begin{proof}
Since these properties for $f_*(\sF)$ are local on $S$ and stable under base change, we may assume $S$ is a derived affine scheme. Let $a \le b$ be integers. To prove~\eqref{lem:relative-perf-1}, it suffices to show:
\begin{itemize}
    \item[(*)] If $\sF$ has relative Tor-amplitude in $[a,b]$ over $S$, then $f_*(\sF)$ has Tor-amplitude in $[a-N, b]$ for some integer $N \ge 0$.
\end{itemize}
For any $\sG \in \Dqc(S)^\heartsuit$, from projection formula we have
\(
    f_*(\sF) \otimes \sG \simeq f_*\big(\sF \otimes f^*(\sG)\big).
\)
Given our assumption on $\sF$, we have $\pi_i(\sF \otimes f^*(\sG)) = 0$ for $i \notin [a, b]$. By \cite[Proposition 2.5.4.4]{SAG} (see also \cite[Proposition 3.9.2]{Lip} for the classical case), there exists $N \ge 0$ so that $f_*$ sends $\Dqc(X)_{\ge a}$ into $\Dqc(S)_{\ge a-N}$. Thus,
\[
    \pi_i\big(f_*(\sF \otimes f^*(\sG))\big) = 0 \quad \text{for } i \notin [a-N, b],
\]
which proves~(*).

If $f$ is also proper and $\sF$ is almost perfect, then $f_*(\sF)$ is almost perfect by \cite[Theorem 5.6.0.2]{SAG} (classically, this is Kiehl's Finiteness Theorem \cite[Corollary 4.3.3.2]{Lip}). Thus, assertion~\eqref{lem:relative-perf-2} follows from~\eqref{lem:relative-perf-1}.
\end{proof}

\subsection{The $\shQ$-complex}
Let $f \colon X \to S$ be a morphism of derived stacks, and let $\mathscr{E}, \mathscr{F} \in \Dqc(X)$. In this subsection, we construct the $\mathcal{Q}$-complex $\shQ_{f}(\mathscr{E}, \mathscr{F})$ under suitable mild assumptions, which serves as a dual to $f_* \sHom_X(\mathscr{E}, \mathscr{F})$.

If $f$ has finite Tor-amplitude and $\mathscr{F}$ is perfect, the $\mathcal{Q}$-complex admits a simple description:
\[
\shQ_{f}(\mathscr{E}, \mathscr{F}) \simeq f_+\big(\mathscr{E} \otimes \mathscr{F}^\vee\big);
\]
see Remark~\ref{rmk:Q:basic}.(3), where $f_+$ denotes the left adjoint to $f^*$ constructed in \cite[\S 6.4.5]{SAG}. In this case, the formal properties of $\shQ_f(\mathscr{E},\mathscr{F})$ follow from those of $f_+$ in  \cite[\S 6.4.5]{SAG}.

Our goal is to establish analogous results without the assumptions of finite Tor-amplitude for $f$ or perfectness of $\mathscr{F}$. 

We first focus on the case where $f$ is a morphism between perfect stacks.

\begin{construction}[The $\mathcal{Q}$-complex: perfect stack case]
\label{constr:Q:perf}
Let $f \colon X \to S$ be a proper, representable morphism between perfect stacks, locally almost of finite presentation. Let $\sE, \sF \in \Dqc(X)$, and assume that $\sF$ is relatively perfect over $S$. Consider the functor
\[
    \shQ_{f}^{\perf}(\blank, \sF) \colon \Perf(X) \to \Dqc(S),
\]
\[
   \mathscr{E} \mapsto f_*(\sE^\vee \otimes \sF)^\vee \in \Perf(S) \subseteq \Dqc(S).
\]
This is well-defined by Lemma~\ref{lem:relative-perf}.\eqref{lem:relative-perf-2}. 
Since $\Dqc(X) \simeq \Ind \Perf(X)$, by the universal property of Ind-completionthe (\cite[Proposition 5.3.5.10]{HTT}), the finite-coproduct-preserving functor $\shQ_{f}^{\perf}(\blank, \sF)$ extends essentially uniquely to a colimit-preserving functor
\[
    \shQ_{f}(\blank, \sF) \colon \Dqc(X) \to \Dqc(S).
\]
\end{construction}

\begin{remark}\label{rmk:Q:basic}
\begin{enumerate}
    \item By definition, there are canonical equivalences
    \[
        \mathcal{Q}_f(\mathscr{E}[1], \mathscr{F}) \simeq \mathcal{Q}_f(\mathscr{E}, \mathscr{F}[-1]) \simeq \mathcal{Q}_f(\mathscr{E}, \mathscr{F})[1].
    \]
    Moreover, for any perfect complex $\mathscr{P} \in \mathrm{Perf}(X)$, we have a canonical equivalence
    \[
        \mathcal{Q}_f(\mathscr{E} \otimes \mathscr{P}, \mathscr{F}) \simeq \mathcal{Q}_f(\mathscr{E}, \mathscr{P}^\vee \otimes \mathscr{F}).
    \]
    \item For any $\mathscr{E} \in \Dqc(X)$, writing $\mathscr{E} = \varinjlim_{\alpha} \mathscr{E}_\alpha$ as a filtered colimits with $\mathscr{E}_\alpha \in \Perf(X)$, there is a canonical equivalence
    \[
        \shQ_{f}(\mathscr{E}, \mathscr{F}) \simeq \varinjlim_{\alpha} \left( f_*(\mathscr{E}_\alpha^\vee \otimes \mathscr{F})^\vee \right).
    \]
   \item If $\mathscr{F} \in \Perf(X)$ and $f$ is locally of finite Tor-amplitude, then $f^*$ admits a left adjoint $f_+$ (see \cite[Proposition 6.4.5.3]{SAG}), and
    \[
        \shQ_{f}(\mathscr{E}, \mathscr{F}) \simeq f_+(\mathscr{E} \otimes \mathscr{F}^\vee).
    \]
\end{enumerate}
\end{remark}

\begin{proposition}
\label{thm:Q}
Assume the setting of Construction~\ref{constr:Q:perf}. Then:
\begin{enumerate}
    \item \label{thm:Q-1} \textbf{(Universal property)}  
    For any $\sG \in \Dqc(S)$, there is a canonical equivalence
	 \begin{equation*}
		\sHom_S\big(\shQ_{f}(\sE,\sF), \sG\big) \simeq f_* \sHom_X\big(\sE ,\sF \otimes f^*(\sG)\big)
	\end{equation*}
    in $\Dqc(S)$, functorial with respect to $\mathscr{G} \in \Dqc(S)$. 
    
    In particular, applying global section functor, we obtain the following universal property for $\shQ_{f}(\sE,\sF)$: for all $\sG \in \Dqc(S)$, there are canonical equivalences
    \[
        \Map_S\big(\shQ_{f}(\sE,\sF), \sG\big) \simeq \Map_X\big(\sE, \sF \otimes f^*\sG\big).
    \]
    
    \item \label{thm:Q-2} \textbf{(Base change)}  
    Given any Cartesian diagram of perfect stacks
   \begin{equation*}
	\begin{tikzcd}
		X' \ar{d}{f'} \ar{r}{g'} & X \ar{d}{f} \\
		S' \ar{r}{g} & S,
	\end{tikzcd}
	\end{equation*}
    there is a canonical equivalence
    \[
       g^*(\shQ_{f}(\sE, \sF)) \simeq \shQ_{f'}(g'^* \sE,g'^* \sF).
    \]

    \item \label{thm:Q-3} \textbf{(Finiteness)}  
    If $\sE$ is almost perfect, then $\shQ_{f}(\sE,\sF)$ is also almost perfect.
\end{enumerate}
\end{proposition}

\begin{proof}
We first prove assertion \eqref{thm:Q-1}. We first prove the universal property~\eqref{thm:Q-1}. For $\sE \in \Perf(X)$, there are canonical equivalences in $\Dqc(S)$:
\begin{align*}
    \sHom_S(\shQ_{f}(\sE,\sF), \sG)
    &\simeq \sHom_S(f_*(\sE^\vee \otimes \sF)^\vee, \sG) \\
    &\simeq f_*(\sE^\vee \otimes \sF) \otimes \sG \\
    &\simeq f_*(\sE^\vee \otimes \sF \otimes f^*\sG) \\
    &\simeq f_*\sHom_X(\sE, \sF \otimes f^*\sG).
\end{align*}
Since $\Dqc(X) \simeq \Ind\Perf(X)$ and both sides commute with filtered colimits in $\sE$, the above equivalence extends to all $\sE \in \Dqc(X)$.
 
Assertion \eqref{thm:Q-2} can be proved similarly as \eqref{thm:Q-1}. Alternatively, it can be deduced from the universal property \eqref{thm:Q-1}: for any $\sH \in \Dqc(S)$, we have canonical equivalences
	\begin{align*}
		&\Map_{S'} (\shQ_{f'}(g'^*\sE,g'^*\sF), \sH) \\
		\simeq & \Map_{X'}(g'^*(\sE),g'^*(\sF) \otimes f'^*(\sH)) &(\text{universal property of $\shQ_{f'}$})\\
		\simeq & \Map_X(\sE, \sF \otimes g'_* f'^*(\sH)) & (\text{$g'^* \dashv g'_*$ and projection formula})\\
		\simeq & \Map_X(\sE, \sF \otimes f^* g_* (\sH)) & (\text{base-change formula}) \\
		\simeq & \Map_S(\shQ_{f}(\sE,\sF), g_* (\sH)) &(\text{universal property of $\shQ_{f}$}) \\
		\simeq & \Map_{S'}(g^*\shQ_{f}(\sE,\sF), \sH) & (\text{$g^* \dashv g_*$}).
	\end{align*} 
By Yoneda's lemma, this gives the desired equivalence.

For the finiteness property~\eqref{thm:Q-3}, by \cite[Proposition 9.1.5.1]{SAG}, it suffices to show that for each $n$, the functor
\[
    \sG \mapsto \Map_S(\shQ_{f}(\sE,\sF), \sG) \simeq \Map_X(\sE, \sF \otimes f^*\sG)
\]
preserves filtered colimits on $\Dqc(S)_{\le n}$. If $\sF$ has relative Tor-amplitude $\le N$ over $S$, then $\sG \mapsto \sF \otimes f^*\sG$ preserves filtered colimits and sends $\Dqc(S)_{\le n}$ to $\Dqc(X)_{\le n+N}$ (see Remark~\ref{rmk:relative-Tor_preserves_truncated}). Since the $t$-structure on $\Dqc(X)$ is compatible with filtered colimits (see \cite[Corollary 9.1.3.2(3)]{SAG}), and $\Map_X(\sE, \blank)$ commutes with filtered colimits on $\Dqc(X)_{\le n+N}$ (as $\sE$ is almost perfect, see \cite[Proposition 9.1.5.1]{SAG}), the claim follows.
\end{proof}

We now construct the $\mathcal{Q}$-complex in the general setting.

\begin{construction}[The $\mathcal{Q}$-complex]
\label{constr:Q}
Let $f \colon X \to S$ be a proper, representable, and locally almost finitely presented morphism of derived stacks, and let $\mathscr{E}, \mathscr{F} \in \Dqc(X)$ with $\mathscr{F}$ relatively perfect over $S$. For any morphism $\eta \colon \Spec A \to S$ from a derived affine scheme, \cite[Proposition 3.19]{BFNjams} ensures that the fiber product $X_A := X \times_S \Spec A$ is a perfect stack.

Let $f_A \colon X_A \to \Spec A$ be the base change of $f$, and let $\mathscr{E}_A$ and $\mathscr{F}_A$ denote the (derived) pullbacks of $\mathscr{E}$ and $ \mathscr{F}$ to $X_A$. By Proposition~\ref{thm:Q}\eqref{thm:Q-2}, the assignment
\[
    (\eta \colon \Spec A \to S) \mapsto \mathcal{Q}_{f_A}(\mathscr{E}_A, \mathscr{F}_A) \in \Dqc(\Spec A)
\]
defines a global quasi-coherent complex $\mathcal{Q}_{f}(\mathscr{E}, \mathscr{F}) \in \Dqc(S)$, where $\mathcal{Q}_{f_A}(\mathscr{E}_A, \mathscr{F}_A)$ is as in Construction~\ref{constr:Q:perf}. Thus, the construction $(\sE, \sF) \mapsto \shQ_f(\sE, \sF)$ defines an exact bifunctor
\[
    \shQ_f(\underline{\hphantom{A}}, \underline{\hphantom{A}}) \colon \Dqc(X) \times \mathrm{Perf}(X/S)^{\mathrm{op}} \to \Dqc(S).
\]

This construction recovers Construction~\ref{constr:Q:perf} when $X$ and $S$ are perfect stacks. Moreover, the formation of $\mathcal{Q}_f(\mathscr{E}, \mathscr{F})$ commutes with arbitrary base change. Finally, it follows from Proposition~\ref{thm:Q}\eqref{thm:Q-3} that if $\mathscr{E}$ is almost perfect, then so is $\mathcal{Q}_f(\mathscr{E}, \mathscr{F})$.\end{construction}

\begin{remark}[\textbf{Universal property of $\mathcal{Q}$-complexes without perfectness}]
\label{rmk:Q:univ}
Consider the setting of Construction~\ref{constr:Q}, and assume that $\mathscr{E}$ is almost perfect and $S$ is a derived geometric stack. Let $\mathrm{dAff}_{\mathrm{Sm}/S}$ denote the $\infty$-category of pairs $(U, \phi)$ with $U = \Spec A$, $A \in \CAlgDelta$, and $\phi \colon U \to S$ a smooth morphism. By \cite[Lemma 5.2.1]{DAG}, we have:
\[
\Dqc(S) \simeq \varprojlim_{(U \to S) \in \mathrm{dAff}_{\mathrm{Sm}/S}} \Dqc(U).\footnote{Note that by \cite[Corollary 3.2.7]{Khan25}, $\Dqc(S)$ can also be computed as a limit over the subcategory where $U \to U'$ are required to be smooth.}
\]
For each $U \to S$ in $\mathrm{dAff}_{\mathrm{Sm}/S}$, write $f_{U} \colon X_{U} = X \times_S U \to U$ the base change of $f$, and denote by $(-)|_{U}$ and $(-)|_{X_{U}}$ the respective pullbacks. Since $\coprod_{(U \to S) \in \mathrm{dAff}_{\mathrm{Sm}/S}} X_{U} \to X$ is a smooth epimorphism, any morphism $\Spec B \to X$ locally factors through some $X_{U} \to X$. Thus:
\[
\Dqc(X) \simeq \varprojlim_{(U \to S) \in \mathrm{dAff}_{\mathrm{Sm}/S}} \Dqc(X_U).
\]
Every morphism $U \to U'$ in $\mathrm{dAff}_{\mathrm{Sm}/S}$ is quasi-smooth and thus has finite Tor-amplitude; the same holds for $X_U \to X_{U'}$. Therefore, for any locally truncated complex $\mathscr{G} \in \Dqc(S)$, \cite[Lemma 6.5.3.7]{SAG} ensures that the formation of the Hom-complexes
\[
\sHom_{U}\big(\mathcal{Q}_{f}(\mathscr{E}, \mathcal{F})|_{U},\, \mathscr{G}|_{U}\big)
\quad\text{and}\quad
\sHom_{X_{U}}\big(\mathscr{E}|_{X_{U}},\, \mathscr{F}|_{X_{U}} \otimes f_U^*(\mathscr{G}|_{U})\big)
\]
commutes with base changes $U \to U'$ and $X_{U} \to X_{U'}$, respectively, for any morphism $U \to U'$ in $\mathrm{dAff}_{\mathrm{Sm}/S}$. Hence, these constructions descend to global objects  
\[\sHom_S\big(\mathcal{Q}_f(\mathscr{E}, \mathcal{F}), \mathscr{G}\big) \quad \text{and} \quad \sHom_X\big(\mathscr{E}, \mathscr{F} \otimes f^*(\mathscr{G})\big)\]
in $\Dqc(S)$ and $\Dqc(X)$, respectively.
Applying the universal property in the affine case ($S = U$) from Proposition~\ref{thm:Q}\eqref{thm:Q-1}, together with Proposition~\ref{thm:Q}\eqref{thm:Q-2}, we obtain functorial equivalences
\[
\sHom_{U}\big(\mathcal{Q}_{f_U}(\mathscr{E}, \mathcal{F})|_{U},\, \mathscr{G}|_{U}\big)
  \simeq f_{U*} \sHom_{X_{U}}\big(\mathscr{E}|_{X_{U}},\, \mathscr{F}|_{X_{U}} \otimes f_U^*(\mathscr{G}|_{U})\big)
\]
for each $U \to S$ in $\mathrm{dAff}_{\mathrm{Sm}/S}$, compatible with base change along morphisms $U \to U'$ in $\mathrm{dAff}_{\mathrm{Sm}/S}$. Consequently, $\mathcal{Q}_{f}(\mathscr{E}, \mathcal{F})$ satisfies the following universal property: for any locally truncated complex $\mathscr{G} \in \Dqc(S)$, there are canonical equivalences
\begin{align*}
\sHom_S\big(\mathcal{Q}_f(\mathscr{E}, \mathcal{F}), \mathscr{G}\big) & \simeq f_* \sHom_X\big(\mathscr{E}, \mathscr{F} \otimes f^*(\mathscr{G})\big).
\end{align*}

\end{remark}

\begin{remark}[Grothendieck's  $Q$-module and Altman--Kleiman's $H$-module]
\label{rmk:H(I,F)}
Let $f \colon X \to S$ be a finitely presented, proper morphism of classical schemes, and let $\mathscr{E}$ and $\mathscr{F}$ be locally finitely presented $\mathscr{O}_X$-modules, with $\mathscr{F}$ flat over $S$. By Corollary~\ref{cor:Ext:fiberwise} below, the complex $\mathcal{Q}_f(\mathscr{E}, \mathscr{F})$ is connective. Thus, for any quasi-coherent $\mathscr{O}_S$-module $\mathscr{G}$, we have
\[
\operatorname{Map}_S\big(\mathcal{Q}_f(\mathscr{E}, \mathscr{F}), \mathscr{G}\big) \simeq \operatorname{Map}_S\big(\pi_0(\mathcal{Q}_f(\mathscr{E}, \mathscr{F})), \mathscr{G}\big).
\]
The universal property of $\mathcal{Q}$-complex implies a canonical isomorphism of classical Hom groups:
\[
\operatorname{Hom}_S\big(\pi_0(\mathcal{Q}_f(\mathscr{E}, \mathscr{F})), \mathscr{G}\big) \simeq \operatorname{Hom}_X(\mathscr{E}, \mathscr{F} \otimes_{\mathscr{O}_S} \mathscr{G}).
\]
Thus, the truncation $\pi_0(\mathcal{Q}_f(\mathscr{E}, \mathscr{F}))$ of the complex $\mathcal{Q}_f(\mathscr{E},\mathscr{F})$ recovers the module $H(\mathscr{E}, \mathscr{F})$ of Altman--Kleiman \cite[(1.1)]{AK80}, generalizing Grothendieck's $Q$-module \cite[III$_2$ Thm.~7.7.6]{EGA}.
\end{remark}

\subsection{Fiberwise criteria for perfectness of $\mathcal{Q}$-complexes}
We now provide fiberwise criteria for the connectivity and Tor-amplitude of the complex $\shQ_{f}(\sE,\sF)$ of Construction~\ref{constr:Q}.

Since the connectivity and Tor-amplitude for $\shQ_{f}(\sE,\sF)$ are local in the flat topology and stable under arbitrary base change, it suffices to consider the case where $S$ is a quasi-compact, separated derived scheme (or even affine). 

For any point $s \in S$, we use the following convention: form the Cartesian square
	\begin{equation*}
	\begin{tikzcd}		
	X_s =X \times_S \{s\} \ar{d}{f_s} \ar{r}{i_{X_s}} & X \ar{d}{f} \\
		\{s\} \ar{r}{i_s} & S.
	\end{tikzcd}
	\end{equation*}
and denote the pullbacks by $\sE|_{X_s} :=i_{X_s}^* \sE$ and $\sF|_{X_s} :=i_{X_s}^* \sF$.
	
\begin{corollary}[Fiberwise criteria of connectivity and perfectness]
\label{cor:Ext:fiberwise}
Let $f \colon X \to S$ be a proper, locally almost finitely presented morphism between quasi-compact, separated derived schemes. Let $\sE, \sF \in \Dqc(X)$, with $\sE$ almost perfect, and $\sF$ relative perfect over $S$. Then:
\begin{enumerate}
	\item 
	\label{cor:Ext:fiberwise-1}
	 (\emph{Connectivity})  
	Let $n$ be an integer. The following are equivalent:
		\begin{enumerate}[label=(\roman*), ref=$\arabic{enumi}\roman*$]
			\item 
			\label{cor:Ext:fiberwise-1i}
      			  $\shQ_{f}(\sE,\sF)$ is $n$-connective, i.e., $\pi_j(\shQ_{f}(\sE,\sF)) = 0$ for all $j < n$.
			\item
			\label{cor:Ext:fiberwise-1ii}
			 For each point $s \in S$, $\Ext_{X_s}^j(\sE|_{X_s}, \sF|_{X_s}) = 0$ for all $j < n$.
			\item
			\label{cor:Ext:fiberwise-1iii}
			 For each closed point $s \in S$, $\Ext_{X_s}^j(\sE|_{X_s}, \sF|_{X_s}) = 0$ for all $j < n$.
		\end{enumerate}
	\item
	\label{cor:Ext:fiberwise-2}
 	(\emph{Perfectness and Tor-amplitude})  
    	Let $a \leq b$ be integers. The following are equivalent:	
	\begin{enumerate}[label=(\roman*), ref=$\arabic{enumi}\roman*$]
			\item 
			\label{cor:Ext:fiberwise-2i}
			$\shQ_{f}(\sE,\sF)$ is perfect with Tor-amplitude in $[a,b]$. 
			\item 
			\label{cor:Ext:fiberwise-2ii}
			For each point $s \in S$, $\Ext_{X_s}^j(\sE|_{X_s}, \sF|_{X_s}) = 0$ for  $j \not\in [a,b]$.
			\item 
			\label{cor:Ext:fiberwise-2iii}
			For each closed point $s \in S$, $\Ext_{X_s}^j(\sE|_{X_s}, \sF|_{X_s}) = 0$ for $j \not \in [a,b]$.
	\end{enumerate}
	If any of the above condition is satisfied, there is a canonical equivalence
		$$f_* \sHom_X(\sE, \sF) \simeq \shQ_{f}(\sE,\sF)^\vee,$$
	    and in particular, $f_* \sHom_X(\sE, \sF)$ is a perfect complex of Tor-amplitude in $[-b, -a]$.
\end{enumerate}
\end{corollary}

\begin{proof}
For any point $s \in S$, let $\kappa(s)$ denote its residue field and consider the Cartesian diagram above. By Proposition~\ref{thm:Q} \eqref{thm:Q-1}\&\eqref{thm:Q-2}, there is a canonical equivalence
\[
    \sHom_{\kappa(s)}(i_s^* \shQ_{f}(\sE,\sF), \kappa(s)) \simeq \sHom_{X_s}(\sE|_{X_s}, \sF|_{X_s}) \in \D(\kappa(s)).
\]
The implications \eqref{cor:Ext:fiberwise-1i} $\Rightarrow$ \eqref{cor:Ext:fiberwise-1ii} and \eqref{cor:Ext:fiberwise-2i} $\Rightarrow$ \eqref{cor:Ext:fiberwise-2ii} follow immediately: the (derived) fiber of $\shQ_{f}(\sE,\sF)$ at $s$ detects the relevant vanishing.

The implications \eqref{cor:Ext:fiberwise-1ii} $\Rightarrow$ \eqref{cor:Ext:fiberwise-1iii} and \eqref{cor:Ext:fiberwise-2ii} $\Rightarrow$ \eqref{cor:Ext:fiberwise-2iii} are tautological. 

For the implications ``$\eqref{cor:Ext:fiberwise-1iii}  \implies \eqref{cor:Ext:fiberwise-1i}$" and  ``$\eqref{cor:Ext:fiberwise-2iii} \implies \eqref{cor:Ext:fiberwise-2i}$", we may assume $S$ is affine. Then, by the equivalence above and the assumptions \eqref{cor:Ext:fiberwise-1iii} or \eqref{cor:Ext:fiberwise-2iii}, we have that $i_s^* \shQ_{f}(\sE,\sF)$ is $n$-connective (resp., has Tor-amplitude in $[a,b]$) for all closed points $s$. Since $\shQ_{f}(\sE,\sF)$ is almost perfect by Proposition~\ref{thm:Q}\eqref{thm:Q-3}, these desired implications follow from \cite[Corollary 2.7.4.3]{SAG}\footnote{Notice that while \cite[Corollary 2.7.4.3]{SAG} proves the assertions under the connectivity assumptions over all points $s$, the same argument applies to closed points.} (for connectivity) and \cite[Corollary 6.1.4.7]{SAG} (for Tor-amplitude). In the classical (underived) setting, these results follow from \cite[\href{https://stacks.math.columbia.edu/tag/068V}{Tag 068V}]{stacks-project}.

Finally, if \eqref{cor:Ext:fiberwise-2i} holds, the equivalence of Proposition \ref{thm:Q}.\eqref{thm:Q-1} applied to $\sG = \sO_S$ gives a canonical equivalence $f_* \sHom_X(\sE, \sF) \simeq \shQ_{f}(\sE,\sF)^\vee$, and hence $f_* \sHom_X(\sE, \sF)$ is perfect of Tor-amplitude in $[-b, -a]$.
\end{proof}

\subsection{Relative Serre duality via the $\mathcal{Q}$-complex}
We now present a version of relative Serre duality expressed in terms of the $\mathcal{Q}$-complex.

Let $f \colon X \to S$ be a morphism of derived stacks. Suppose that $f$ is proper, representable, locally almost of finite presentation, and locally of finite Tor-amplitude. Under these conditions, the pushforward functor $f_*$ admits a right adjoint $f^! \colon \Dqc(S) \to \Dqc(X)$ (\cite[Theorem 3.8.(2a)]{J22a}; cf. \cite[Proposition 6.4.2.1]{SAG}). We define the \emph{relative dualizing complex} as 
	\[\omega_f^\bullet := f^!(\mathscr{O}_S) \in \Dqc(X).\] 
The formation of $f^!$ and $\omega_f^\bullet$ commutes with arbitrary base change, and there is a canonical equivalence of functors:
\[
f^!(\underline{\hphantom{A}}) \simeq f^*(\underline{\hphantom{A}}) \otimes \omega_f^\bullet.
\]
See \cite[\S 6.4.2]{SAG} (or \cite[Theorem 3.8.(2b)\&(2c)]{J22a}) for more details.

\begin{proposition}[Relative Serre duality]
\label{prop:Serre:Q}
Let $f \colon X \to S$ be a morphism of derived stacks that is proper, representable, locally almost of finite presentation, and locally of finite Tor-amplitude.
\begin{enumerate}
    \item  
    \label{prop:Serre:Q-1} 
    The relative dualizing complex $\omega_f^\bullet := f^!(\mathscr{O}_S)$ is relatively perfect over $S$.
    \item 
    \label{prop:Serre:Q-2}
    The canonical map $\mathcal{Q}_{f}(\omega_f^\bullet, \omega_f^\bullet) \to \mathscr{O}_S$ is an equivalence.
    \item   
    \label{prop:Serre:Q-3}
    For any pair $\mathscr{E}, \mathscr{F} \in \Dqc(X)$, with $\mathscr{E}$ perfect and $\mathscr{F}$ relatively perfect over $S$, the complexes $\mathcal{Q}_f(\mathscr{E},\mathscr{F})$ and $\mathcal{Q}_{f}(\mathscr{F}, \mathscr{E} \otimes \omega_f^\bullet)$ are perfect, and there is a functorial equivalence
    \[
    \mathcal{Q}_{f}(\mathscr{F}, \mathscr{E} \otimes \omega_f^\bullet) \simeq \mathcal{Q}_{f}(\mathscr{E}, \mathscr{F})^\vee.
    \]
\end{enumerate}
\end{proposition}

\begin{proof}
For \eqref{prop:Serre:Q-1}, we may assume $S$ is affine. By \cite[Corollary 3.4.2.3]{SAG} (or \cite[Proposition 3.9.2]{Lip} in the classical case), there exists $N \geq 0$ such that $f_{*}$ sends $\Dqc(X)_{\geq 0}$ into $\Dqc(S)_{\geq -N}$. It follows that its right adjoint satisfies:
\[
f^! \simeq f^*(\underline{\hphantom{A}}) \otimes \omega_f^\bullet \colon \Dqc(S)_{\leq 0} \to \Dqc(X)_{\leq N}.
\]
Since $\omega_f^\bullet$ is almost perfect (\cite[Proposition 6.4.4.1]{SAG}), and by Remark~\ref{rmk:relative-Tor_preserves_truncated}, we conclude that $\omega_f^\bullet$ is relatively perfect .

For \eqref{prop:Serre:Q-3}, let $\eta \colon T \to S$ be any morphism from an affine derived scheme, and denote by $f_T \colon X_T \to T$ the base change. Let $\mathscr{E}_T$ and $\mathscr{F}_T$ be the pullbacks to $X_T$. For any $\mathscr{G} \in \Dqc(T)$, we have functorial equivalences:
\begin{align*}
\Map_{X_T}\big(\mathscr{F}_T, \mathscr{E}_T \otimes \omega_{f_T}^\bullet \otimes f_T^*(\mathscr{G}) \big)
&\simeq \Map_{X_T}\big(\mathscr{F}_T, \mathscr{E}_T \otimes f_T^!(\mathscr{G})\big) \\
&\simeq \Map_{X_T}\big(\mathscr{E}_T^\vee \otimes \mathscr{F}_T, f_T^!(\mathscr{G})\big) \\
&\simeq \Map_T\big(f_{T*}(\mathscr{E}_T^\vee \otimes \mathscr{F}_T), \mathscr{G}\big) \\
&\simeq \Map_T\big(\mathcal{Q}_{f_T}(\mathscr{E}_T, \mathscr{F}_T)^\vee, \mathscr{G}\big) 
\end{align*}
Thus, there is a canonical equivalence
\[
\mathcal{Q}_{f_T}(\mathscr{F}_T, \mathscr{E}_T \otimes \omega_{f_T}^\bullet) \simeq \mathcal{Q}_{f_T}(\mathscr{E}_T, \mathscr{F}_T)^\vee \in \mathrm{Perf}(T),
\]
compatible with arbitrary base change in $T$. The result follows.

For \eqref{prop:Serre:Q-2}, we apply \eqref{prop:Serre:Q-3} twice to obtain canonical equivalences:
\[
\mathcal{Q}_{f}(\omega_f^\bullet, \mathscr{O}_X \otimes \omega_f^\bullet)
\simeq \mathcal{Q}_{f}(\mathscr{O}_X, \omega_f^\bullet)^\vee
\simeq \mathcal{Q}_{f}(\mathscr{O}_X, \mathscr{O}_X)
\simeq \mathscr{O}_S.
\]
\end{proof}

\begin{remark}
In the context of Proposition~\ref{prop:Serre:Q}, let $a \leq b$, and assume $S$ is a derived scheme. The conditions \eqref{cor:Ext:fiberwise-2i}--\eqref{cor:Ext:fiberwise-2iii} of Corollary~\ref{cor:Ext:fiberwise}\eqref{cor:Ext:fiberwise-2} are equivalent, via the relative Serre duality, to:
\begin{enumerate}
    \item[$(i')$] $\mathcal{Q}_{f}(\mathscr{F}, \mathscr{E} \otimes \omega_f^\bullet)$ is perfect with Tor-amplitude in $[-b, -a]$.
    \item[$(ii')$] For each $s \in S$, $\Ext_{X_s}^j\big(\mathscr{F}|_{X_s}, \mathscr{E}|_{X_s} \otimes \omega_{X_s}^\bullet\big) = 0$ for $j \notin [-b, -a]$.
    \item[$(iii')$] For each closed point $s \in S$, $\Ext_{X_s}^j\big(\mathscr{F}|_{X_s}, \mathscr{E}|_{X_s} \otimes \omega_{X_s}^\bullet\big) = 0$ for $j \notin [-b, -a]$.
\end{enumerate}
Here, $\omega_{X_s}^\bullet := f_s^!(\mathscr{O}_{\kappa(s)})$ is the dualizing complex for the fiber $f_s \colon X_s \to \Spec \kappa(s)$.

This is compatible with Serre duality on fibers: for all $s \in S$ and for all $j \in \mathbb{Z}$, there are natural isomorphisms of $\kappa(s)$-vector spaces:
\[
\Ext_{X_s}^j\big(\mathscr{F}|_{X_s}, \mathscr{E}|_{X_s} \otimes \omega_{X_s}^\bullet\big) \simeq \Ext_{X_s}^{-j}\big(\mathscr{E}|_{X_s}, \mathscr{F}|_{X_s}\big)^\vee.
\]
\end{remark}

\subsection{Relative coherent duality via the $\mathcal{Q}$-complex}
We now formulate a version of relative coherent duality in terms of the $\mathcal{Q}$-complex.

\begin{proposition}[Relative coherent duality]
\label{prop:coh:duality}
Let $f \colon X \to S$ be a morphism of derived geometric stacks that is proper, representable, locally almost of finite presentation, and locally of finite Tor-amplitude. Assume further that $\omega_f^\bullet$ is locally truncated (e.g., if $\mathscr{O}_S$ is locally truncated).
\begin{enumerate}
    \item For any almost perfect complex $\mathscr{F}$, the internal Hom-complex
    \[
    \mathbb{D}(\mathscr{F}) := \sHom_{X}(\mathscr{F}, \omega_f^\bullet)
    \]
    defines a quasi-coherent complex on $X$, which is computed ``smooth locally" by
    \[
    \mathbb{D}(\mathscr{F})|_V = \sHom_{T}(\mathscr{F}|_V, (\omega_f^\bullet)|_V)
    \]
    for $V \to X$ ranging over either of the following categories: all smooth morphisms with $V$ derived affine, or $V = X \times_S U \to X$ where $U \to S$ is smooth and $U$ is derived affine.
    
    \item For any $\mathscr{E}, \mathscr{F} \in \Dqc(X)$, with $\mathscr{E}$ perfect and $\mathscr{F}$ relatively perfect over $S$, there is a functorial equivalence in $\Dqc(S)$:
    \[
    \mathcal{Q}_{f}\big(\mathbb{D}(\mathscr{F}), \mathbb{D}(\mathscr{E})\big) \xrightarrow{\simeq} \mathcal{Q}_{f}(\mathscr{E}, \mathscr{F}).
    \]
\end{enumerate}
\end{proposition}

\begin{proof}
Part (1) follows by applying the argument of Remark~\ref{rmk:Q:univ} to the almost perfect complex $\mathscr{F}$ and the locally truncated complex $\omega_f^\bullet$.

For part (2), we first consider the case where $S$ is derived affine. By the definition of the internal Hom complex, the identity map $\mathbb{D}(\mathscr{F}) \to \mathbb{D}(\mathscr{F})$ corresponds to a canonical map $\mathscr{F} \otimes \mathbb{D}(\mathscr{F}) \to \omega_{f}^\bullet$. Thus, for any $\mathscr{G} \in \Dqc(S)$, there is an induced canonical map
\[
(\mathscr{E}^\vee \otimes f^*\mathscr{G}) \otimes \mathscr{F} \otimes \mathbb{D}(\mathscr{F}) 
\to 
(\mathscr{E}^\vee \otimes f^* \mathscr{G}) \otimes \omega_{f}^\bullet,
\]
which in turn induces
\[
\mathscr{E}^\vee \otimes \mathscr{F} \otimes f^*\mathscr{G} 
\to 
\sHom_X\left( \mathbb{D}(\mathscr{F}),\, \mathscr{E}^\vee \otimes \omega_{f}^\bullet \otimes f^* \mathscr{G} \right).
\]
This yields a functorial morphism
\begin{equation}\label{eqn:coh.duality}
f_* \sHom_X \left(\mathscr{E}, \mathscr{F} \otimes f^* \mathscr{G}\right)
\to 
f_* \sHom_X\left(\mathbb{D}(\mathscr{F}),\, \mathbb{D}(\mathscr{E}) \otimes f^* \mathscr{G} \right)
\end{equation}
which, by the universal property Proposition~\ref{thm:Q}~\eqref{thm:Q-1}, induces a canonical map 
\[
\mathcal{Q}_f\big(\mathbb{D}(\mathscr{F}), \mathbb{D}(\mathscr{E})\big) \to \mathcal{Q}_f(\mathscr{E},\mathscr{F}).
\] 
It remains to show that the morphism in~\eqref{eqn:coh.duality} is an isomorphism. First, we may assume $\mathscr{E} = \mathscr{O}_X$. Since $f_*$ preserves small colimits, and the formation of~\eqref{eqn:coh.duality} commutes with small colimits in the variable $\mathscr{G}$, we may assume $\mathscr{G}$ is perfect. As this also commutes with cones and direct sums in the variable  $\mathscr{G}$, we reduce to $\mathscr{G} = \mathscr{O}_S$. A similar argument reduces to the case $\mathscr{F} = \mathscr{O}_X$. In this situation, the desired isomorphism follows from the canonical isomorphism
\[
\mathcal{Q}_f(\omega_f^\bullet, \omega_f^\bullet) \to \mathscr{O}_S
\]
given by Proposition~\ref{prop:Serre:Q}~\eqref{prop:Serre:Q-2}.

For general $S$, since both sides of the canonical morphism $\mathcal{Q}_f\big(\mathbb{D}(\mathscr{F}), \mathbb{D}(\mathscr{E})\big) \to \mathcal{Q}_f(\mathscr{E},\mathscr{F})$
can be computed over smooth morphisms $U \to S$ with $U$ derived affine, and the formation is compatible with base change along morphisms $U \to U'$ in $ \mathrm{Aff}_{\mathrm{Sm}/S}$, the result follows from the derived affine case.
\end{proof}

The canonical map $\mathscr{F} \to \mathbb{D}(\mathbb{D}(\mathscr{F}))$ need not to be an equivalence in general. It is an equivalence, however, if $S$ is Gorenstein (i.e., if $S$ possesses an invertible absolute dualizing complex).

\subsection{Simplicity via the $\mathcal{Q}$-complex}
\label{subsec:simple}
Let $f \colon X \to S$ be a proper, representable, and locally almost finitely presented morphism of derived stacks, and let $\mathscr{E} \in \Dqc(X)$ be relatively perfect over $S$. The identity morphism $\mathscr{E} \to \mathscr{E}$ gives rise to a canonical map
\[
\mathcal{Q}_{f}(\mathscr{E}, \mathscr{E}) \to \mathscr{O}_S.
\]
We say that $\mathscr{E}$ is \emph{$S$-simple} if the fiber of this map is $1$-connective, i.e., its homotopy groups vanish in degrees $<1$. By Corollary~\ref{cor:Ext:fiberwise}, this condition admits the following  characterization:

\begin{lemma}
\label{lem:simple}
In the above setting, $\mathscr{E}$ is $S$-simple if and only if, for any $s \in S$, the natural map
\[
\kappa(s) \to\RR \operatorname{Hom}_{X_s}(\mathscr{E}|_{X_s}, \mathscr{E}|_{X_s})
\]
induces an isomorphism $\kappa(s) \simeq \mathrm{Ext}_{X_s}^0(\mathscr{E}|_{X_s}, \mathscr{E}|_{X_s})$, and $\mathrm{Ext}_{X_s}^j(\mathscr{E}|_{X_s}, \mathscr{E}|_{X_s}) = 0$ for all $j < 0$.
\end{lemma}

Thus, our definition agrees with that of To\"en--Vaqui\'e \cite[Definition 5.2]{TV08}. Together with \cite[Lemma 2.36]{Mon25}, this shows that our definition is also compatible with Montagnani's \cite[Definition 2.34]{Mon25} in the non-archimedean setting. When $\mathcal{Q}_{f}(\mathscr{E}, \mathscr{E})$ is perfect, this equivalence follows immediately from Corollary~\ref{cor:Ext:fiberwise}\eqref{cor:Ext:fiberwise-2}: taking duals, our notion of $\mathscr{E}$ being $S$-simple is equivalent to requiring that the cofiber of $\mathscr{O}_S \to f_* \sHom_{X}(\mathscr{E}, \mathscr{E})$ has Tor-amplitude in strictly negative degrees, which is precisely the condition of \cite[Definition 2.34]{Mon25}.

\subsection{Cotangent complexes via the $\mathcal{Q}$-complex}

Let $X$ be a proper scheme over $\mathbb{C}$. Following \cite{PS}, consider the derived moduli stack $\underline{\mathbf{Coh}}(X)$, which assigns to each derived affine $\mathbb{C}$-scheme $T$ the full space spanned by almost perfect complexes on $X \times_{\mathbb{C}} T$ with relative Tor-amplitude $\le 0$ over $T$. Let
\[
f \colon \mathscr{X} := \underline{\mathbf{Coh}}(X) \times_{\mathbb{C}} X \to \underline{\mathbf{Coh}}(X)
\]
be the projection, and let $\mathscr{U} \in \Dqc(\mathscr{X})$ be the universal almost perfect complex.

Using the $\mathcal{Q}$-complex of Construction~\ref{constr:Q}, we obtain an explicit description of the cotangent complex of $\underline{\mathbf{Coh}}(X)$, whose existence was established by Porta--Sala \cite[Lemma 2.21]{PS}:

\begin{lemma}
\label{lem:cotangent:Coh}
In the above setting, the derived stack $\underline{\mathbf{Coh}}(X)$ admits an almost perfect cotangent complex over~$\mathbb{C}$, given explicitly by
\[
\mathbb{L}_{\underline{\mathbf{Coh}}(X)/\mathbb{C}} \simeq \mathcal{Q}_{f}(\mathscr{U}, \mathscr{U}[1]).
\]
\end{lemma}

In particular, the connectivity and Tor-amplitude of $\mathbb{L}_{\underline{\mathbf{Coh}}(X)/\mathbb{C}}$ are determined by the fiberwise criteria of Corollary~\ref{cor:Ext:fiberwise}.

\begin{proof}
Let $T = \operatorname{Spec} (A)$, where $A \in \CAlgDelta_{\mathbb{C}}$, and $\eta \in \underline{\mathbf{Coh}}(X)(A)$, corresponding to an almost perfect complex $\mathscr{F}_{T}$ on $X \times_{\mathbb{C}} T$ of relative Tor-amplitude $\le 0$ over $T$. To compute the cotangent complex, as in \cite[\S 3.2]{DAG}, consider the functor $F_\eta \colon \Mod_A^\cn \to \shS$,
\[
F_\eta(M) = \mathrm{fib}\big(\underline{\mathbf{Coh}}(X)(A \oplus M) \to \underline{\mathbf{Coh}}(X)(A)\big).
\]
The proof of \cite[Lemma 2.21]{PS} shows that this functor is computed by the mapping space
\[
F_{\eta}(M) \simeq \Map_{X \times_{\mathbb{C}} T}\big(\mathscr{F}_{T}, \mathscr{F}_{T}[1] \otimes p_T^*M\big).
\]
The universal property of the $\mathcal{Q}$-complex (Proposition~\ref{thm:Q}(\ref{thm:Q-1})) gives
\[ F_{\eta}(M) \simeq \Map_T\big(\mathcal{Q}_{f_T}(\mathscr{F}_{T}, \mathscr{F}_{T}[1]), M\big).\]
Thus, $F_\eta$ is corepresented by the almost perfect complex $\mathcal{Q}_{f_T}(\mathscr{F}_{T}, \mathscr{F}_{T}[1])$, whose formation commutes with arbitrary base change (Proposition~\ref{thm:Q}(\ref{thm:Q-2})), yielding the result.
\end{proof}

\begin{remark}
This argument applies to any proper, presentable, and locally almost finitely presented morphism $X \to S$, provided that the derived stack $\underline{\mathbf{Coh}}(X/S)$ \footnote{The derived moduli stack $\underline{\mathbf{Coh}}(X/S)$ assigns to each morphism $T = \operatorname{Spec} (A) \to S$ the full subspace spanned by almost perfect complexes on $X_T$ with relative Tor-amplitude $\le 0$ over $T$.} is infinitesimally cohesive. In this case, $\underline{\mathbf{Coh}}(X/S)$ admits an almost perfect relative cotangent complex over $S$:
\[
\mathbb{L}_{\underline{\mathbf{Coh}}(X/S)/S} \simeq \mathcal{Q}_{f}(\mathscr{U}, \mathscr{U}[1]),
\]
where $f \colon X \times_S \underline{\mathbf{Coh}}(X/S) \to \underline{\mathbf{Coh}}(X/S)$ is the projection and $\mathscr{U}$ is the universal complex.
\end{remark}

\section{Moduli of extensions}
\label{sec:extension}

Using the $\mathcal{Q}$-complex constructed in the previous section, together with the theory of derived affine cones, projectivizations, and Grassmannians \cite{J22a, J22b, J23}, we develop a natural derived framework for moduli stacks of extensions. This perspective provides a robust generalization of the classical constructions of \cite{Lag83} to the derived setting.

In this section, we fix a morphism of derived stacks
\[
f \colon \mathscr{X} \to \mathscr{M}
\]
which is proper, representable, locally almost of finite presentation, and locally of finite Tor-amplitude. We adopt this notation, rather than $X \to S$, since our main applications concern the case where $f$ is the universal family over a moduli stack $\mathscr{M}$.

For $\mathscr{E}, \mathscr{F} \in \Dqc(\mathscr{X})$ and any morphism $\eta \colon T \to \mathscr{M}$, we denote by $f_T \colon X_T \to T$ the base change of $f$ along $\eta$, where $X_T := \mathscr{X} \times_{\mathscr{M}} T$ is the derived fiber product. We write $\mathscr{E}_T$ and $\mathscr{F}_T$ for the derived pullbacks of $\mathscr{E}$ and $\mathscr{F}$ to $X_T$. Note that when $\mathscr{X}/\mathscr{M}$ is flat and $\mathscr{M}$ and $T$ are classical, $X_T$ coincides with the classical fiber product $\mathscr{X} \times_{\mathscr{M}}^{\mathrm{cl}} T$.

\subsection{Linear stacks of extensions $\underline{\mathbf{Ext}}_f^{i}$}
The constructions and results in this subsection, though not required for the subsequent applications in this paper, fit naturally into the overall framework developed here and hold promise for future applications.

The \emph{linear stack} $\mathbf{V}(\mathcal{G})$ of a complex $\mathcal{G} \in \Dqc(\mathscr{M})$ over a derived stack $\mathscr{M}$ (see \cite[\S 3.3]{ToenDAG}, \cite[\S 4.1]{J22a} or \cite[\S 2]{HKR}) is the derived moduli stack which, for any morphism $T \to \mathscr{M}$ from a derived affine scheme $T$, assigns the space (i.e., $\infty$-groupoid) of cosections $\mathcal{G}_T \to \mathscr{O}_T$:
\[
\mathbf{V}(\mathcal{G})(T) \simeq \Map_T( \mathcal{G}_T, \mathscr{O}_T).
\]

\begin{definition}
Let $\mathscr{E}, \mathscr{F} \in \Dqc(\mathscr{X})$, with $\mathscr{F}$ relatively perfect over $\mathscr{M}$. Let $\shQ_{f}(\sE,\sF)$ denote the complex of Construction~\ref{constr:Q}. For any integer $i \in \mathbb{Z}$, we define the linear stack
\[
\underline{\mathbf{Ext}}^i_f(\mathscr{E}, \mathscr{F}) := \mathbf{V}(\mathcal{Q}_{f}(\mathscr{E},\mathscr{F}[i])) \to \mathscr{M}.
\]
Explicitly, the functor of points of $\underline{\mathbf{Ext}}^i_f(\mathscr{E}, \mathscr{F})$ is described as follows: for any morphism $T \to \mathscr{M}$ from a derived affine scheme $T$, by Proposition~\ref{thm:Q}~\eqref{thm:Q-1}\&\eqref{thm:Q-2}, we have
\begin{align*}
\underline{\mathbf{Ext}}^i_f(\mathscr{E}, \mathscr{F})(T)
&= \Map_{T}\big( \mathcal{Q}_{f}(\mathscr{E},\mathscr{F}[i])_T, \mathscr{O}_T\big) \\
&\simeq \Map_{T}\big(\mathcal{Q}_{f}(\mathscr{E}_T,\mathscr{F}_T[i]), \mathscr{O}_T\big) \\
&\simeq \Map_{X_T}\big(\mathscr{E}_T, \mathscr{F}_T[i]\big),
\end{align*}
i.e., $\underline{\mathbf{Ext}}^i_f(\mathscr{E}, \mathscr{F})(T)$ is the space of homomorphisms from $\mathscr{E}_T$ to $\mathscr{F}_T[i]$ in $\Dqc(X_T)$. In particular, the set of connected components recovers the usual Ext-group:
\[
\pi_0\big(\underline{\mathbf{Ext}}_f^i (\mathscr{E},\mathscr{F})(T)\big) = \mathrm{Ext}_{X_T}^i(\mathscr{E}_T, \mathscr{F}_T).
\]
Thus, $\underline{\mathbf{Ext}}_f^i(\mathscr{E}, \mathscr{F})$ is the derived moduli stack parametrizing $i$th extensions of $\mathscr{E}$ by $\mathscr{F}$ over $\mathscr{M}$.
\end{definition}

The following properties are immediate from those of linear stacks (\cite[Proposition 4.10]{J22a}, \cite[Proposition 2.11]{HKR}; see also \cite[\S 3.3]{ToenDAG}):
\begin{lemma} For any given $i \in \mathbb{Z}$:
\begin{enumerate}
    \item The formation of $\underline{\mathbf{Ext}}^i_f(\mathscr{E}, \mathscr{F}) \to \mathscr{M}$ commutes with arbitrary base change in $\mathscr{M}$.
    \item If $\shQ_{f}(\sE,\sF)$ is $i$-connective, $\underline{\mathbf{Ext}}^i_f(\mathscr{E}, \mathscr{F}) \to \mathscr{M}$ is a relative derived affine scheme. If, in addition, $\mathscr{E}$ is almost perfect, $\underline{\mathbf{Ext}}^i_f(\mathscr{E}, \mathscr{F}) \to \mathscr{M}$ is locally almost of finite presentation.
    \item If $\shQ_{f}(\sE,\sF)$ is $n$-connective for some $n$, then $\underline{\mathbf{Ext}}^i_f(\mathscr{E}, \mathscr{F}) \to \mathscr{M}$ admits a relative cotangent complex $\mathbb{L}_{\underline{\mathbf{Ext}}^i_f(\mathscr{E}, \mathscr{F})/ \mathscr{M}}$, which is $(n-i)$-connective and given by the pullback of $\mathcal{Q}_{f}(\mathscr{E},\mathscr{F}[i]) \simeq \shQ_{f}(\sE,\sF)[-i]$ along the projection $\underline{\mathbf{Ext}}^i_f(\mathscr{E}, \mathscr{F}) \to \mathscr{M}$.
    \item If $\shQ_{f}(\sE,\sF)$ is perfect of Tor-amplitude in $[a, b]$, then $\mathbb{L}_{\underline{\mathbf{Ext}}^i_f(\mathscr{E}, \mathscr{F})/ \mathscr{M}}$ is perfect of Tor-amplitude in $[a-i, b-i]$. In particular, $\underline{\mathbf{Ext}}^i_f(\mathscr{E}, \mathscr{F}) \to \mathscr{M}$ is smooth if $b \le i$, and quasi-smooth if $b \le i+1$.
\end{enumerate}
\end{lemma}

\begin{variant}[Rigidified moduli of extensions for simple complexes]
Assume the setting of \S\ref{subsec:simple}, and let $\mathscr{E}$ be an $\mathscr{M}$-simple complex in $\Dqc(\mathscr{X})$. Consider the fiber of the canonical map
\[
\mathcal{Q}_{f}(\mathscr{E}, \mathscr{E})^{\mathrm{red}} := \mathrm{fib}\big(\mathcal{Q}_{f}(\mathscr{E}, \mathscr{E}) \to \mathscr{O}_\mathscr{M}\big).
\]
Then $\mathcal{Q}_{f}(\mathscr{E}, \mathscr{E})^{\mathrm{red}}[-1]$ is a connective, almost perfect complex. We consider the affine cone 
\[
\underline{\mathbf{Ext}}_f^1(\mathscr{E}, \mathscr{E})^{\mathrm{rig}} := \mathbf{V}\big(\mathcal{Q}_{f}(\mathscr{E}, \mathscr{E})^{\mathrm{red}}[-1]\big) \to \mathscr{M},
\]
which is a relative derived affine scheme over $\mathscr{M}$, whose formation commutes with arbitrary base change in $\mathscr{M}$.
Moreover, there is a Cartesian square
\begin{equation*}
\begin{tikzcd}
    \underline{\mathbf{Ext}}_f^1(\mathscr{E}, \mathscr{E}) \arrow[r] \arrow[d] & \underline{\mathbf{Ext}}_f^1(\mathscr{E}, \mathscr{E})^{\mathrm{rig}} \arrow[d] \\
    \mathscr{M} \arrow[r] & B_\mathscr{M}^2 \mathbb{G}_a := \mathbf{V}(\mathscr{O}_\mathscr{M}[-2]).
\end{tikzcd}
\end{equation*}
Here, $\mathscr{M} \to B_\mathscr{M}^2 \mathbb{G}_a$ is a smooth, effective epimorphism with fibers $B\mathbb{G}_a$. Consequently, the canonical map
\[
\underline{\mathbf{Ext}}_f^1(\mathscr{E}, \mathscr{E}) \to \underline{\mathbf{Ext}}_f^1(\mathscr{E}, \mathscr{E})^{\mathrm{rig}}
\]
a smooth, effective epimorphism with fibers $B\mathbb{G}_a$ (i.e., it is a $\mathbb{G}_a$-gerbe).
\end{variant}

\subsection{Moduli of classes of non-split extensions $\mathbf{PExt}_f^i$}
\label{subsec:PExt}
The \emph{derived projectivization} $\mathbf{P}(\mathcal{G})$ of a connective complex $\mathcal{G} \in \Dqc(\mathscr{M})$ over a derived stack $\mathscr{M}$ (see \cite[\S 4.2]{J22a}) is the derived moduli stack assigning to each morphism $\eta \colon T \to \mathscr{M}$, where $T$ is a derived affine scheme, the space of pairs $(u, \mathscr{L})$, where $\mathscr{L} \in \mathrm{Pic}(T)$, and $u : \eta^* \mathcal{G} \to \mathscr{L}$ is a morphism in $\Dqc(T)$ that is surjective on $\pi_0$ (i.e., a surjection on zeroth sheaf homology when $T$ is classical).
\begin{definition}
\label{defn:PExt^i}
Let $\mathscr{E}, \mathscr{F} \in \Dqc(\mathscr{X})$, with $\mathscr{E}$ almost perfect and $\mathscr{F}$ relatively perfect over $\mathscr{M}$.

Fix $i \in \mathbb{Z}$, and suppose $\mathcal{Q}_{f}(\mathscr{E},\mathscr{F})$ is $i$-connective. By Corollary~\ref{cor:Ext:fiberwise}, this is equivalent to:
\begin{itemize}
    \item For every point $s \in \mathscr{M}$, $\Ext_{X_s}^j(\mathscr{E}|_{X_s}, \mathscr{F}|_{X_s}) = 0$ for all $j < i$.
\end{itemize}

We define the moduli stack
\[
\pi \colon \mathbf{PExt}_f^i(\mathscr{E},\mathscr{F}) := \mathbf{P}(\mathcal{Q}_{f}(\mathscr{E},\mathscr{F}[i])) \to \mathscr{M},
\]
where $\mathbf{P}(\mathcal{G})$ denotes the derived projectivization of $\mathcal{G}$ over $\mathscr{M}$. Let $\mathscr{O}(1)$ be the universal line bundle on $\mathbf{PExt}_f^i(\mathscr{E},\mathscr{F})$. By construction, we have
\[
\mathbf{PExt}_f^i(\mathscr{E}, \mathscr{F}) = \mathbf{PExt}_f^{i+m}(\mathscr{E}, \mathscr{F}[-m]) \quad \text{for all $m \ge 0$.}
\]
\end{definition}

The functor of points of $ \mathbf{PExt}_f^i(\mathscr{E},\mathscr{F})$ are explicitly given as follows: for any morphism $\eta \colon T \to \mathscr{M}$ from a derived affine scheme (or  a perfect stack) $T$, $\mathbf{PExt}_f^i(\mathscr{E},\mathscr{F})(T)$ is the space of pairs $(u, \mathscr{L})$, where $\mathscr{L} \in \mathrm{Pic}(T)$ and $u : \eta^*\mathcal{Q}_{f}(\mathscr{E},\mathscr{F}[i]) \to \mathscr{L}$ is a morphism in $\Dqc(T)$ that is surjective on $\pi_0$. By Proposition~\ref{thm:Q}, such a morphism $u$ corresponds to a morphism in $\Dqc(X_T)$,
\[
v \colon \mathscr{E}_T \to \mathscr{F}_T[i] \otimes f_T^*\mathscr{L}.
\]
The surjectivity of $u$ on $\pi_0$ is equivalent, by Nakayama's Lemma, to the condition that for any point $t \in T$, the induced map $v|_{X_t} : \mathscr{E}|_{X_t} \to \mathscr{F}|_{X_t}[i]$ is nonzero. Indeed, for each $t \in T$,
\[
\mathbb{R}\operatorname{Hom}_{\kappa(t)}\left( i_t^*\mathcal{Q}_{f}(\mathscr{E},\mathscr{F}[i]), \kappa(t)\right) \simeq\RR \operatorname{Hom}_{X_t}\left(\mathscr{E}|_{X_t}, \mathscr{F}|_{X_t}[i]\right),
\]
where we use the Notation of Corollary~\ref{cor:Ext:fiberwise}. Therefore, the space of pairs $(u,\mathscr{L})$ as above is equivalent to the space of pairs $(v, \mathscr{L})$, where $v$ is nonzero when evaluated at every point.

We summarize this description as follows: 

\begin{definition}
\label{defn:fiberwise.nonzero}
A homomorphism $\phi \colon \mathscr{E} \to \mathscr{F}$ in $\Dqc(\mathscr{X})$ is said to be \emph{fiberwise nonzero over $\mathscr{M}$} if, for every point $s \in \mathscr{M}$, the induced map on the fiber $X_s = \mathscr{X} \times_\mathscr{M} \{s\}$,
\[
\phi|_{X_s} \colon \mathscr{E}|_{X_s} \to \mathscr{F}|_{X_s},
\]
is nonzero, where $\mathscr{E}|_{X_s}$ and $\mathscr{F}|_{X_s}$ denote the pullbacks to the fiber $X_s$ as in Corollary~\ref{cor:Ext:fiberwise}.
\end{definition}

\begin{lemma}
\label{lem:PExt:T-points}
In the context of Definition~\ref{defn:PExt^i}, let $\eta \colon T \to \mathscr{M}$ be a morphism from a perfect stack $T$. Then the space of $T$-points of $\mathbf{PExt}_f^i(\mathscr{E},\mathscr{F})$ (over $\eta$) consists of pairs $(v, \mathscr{L})$, where $\mathscr{L} \in \mathrm{Pic}(T)$ is a line bundle, and
\[
v \colon \mathscr{E}_T \to \mathscr{F}_T[i] \otimes f_T^*\mathscr{L}
\]
is homomorphism in $\Dqc(X_T)$ that is fiberwise nonzero over $T$. 
\end{lemma}

For concreteness, let us consider two examples. For simplicity of exposition, assume that $\mathscr{M}$ is a perfect stack and $\shQ_{f}(\sE,\sF)$ is an $i$-connective perfect complex. Let $P = \mathbf{PExt}_f^i(\mathscr{E},\mathscr{F})$, and let $f_P \colon X_P \to P$ denote the base change of $\mathscr{X} \to \mathscr{M}$ along $P \to \mathscr{M}$.

\begin{example}[$\mathbf{PHom}_f(\mathscr{E},\mathscr{F})$]
\label{eg:PHom}
The stack $P =\mathbf{PHom}_f(\mathscr{E},\mathscr{F}) := \mathbf{PExt}_f^0(\mathscr{E},\mathscr{F})$ is the derived moduli stack parametrizing fiberwise nonzero homomorphisms $\mathscr{E} \to \mathscr{F}$, up to twisting by line bundles from $\mathscr{M}$. There is a universal fiberwise nonzero homomorphism in $\Dqc(X_P)$:
\[
v_{\mathrm{univ}} \colon \mathscr{E}_{P} \to \mathscr{F}_{P} \otimes f_P^*\mathscr{O}(1).
\]
\end{example}

\begin{example}[$\mathbf{PExt}_f^1(\mathscr{E},\mathscr{F})$]
The stack $P = \mathbf{PExt}_f^1(\mathscr{E},\mathscr{F})$ is the \emph{derived moduli stack of classes of fiberwise nonsplit extensions of $\mathscr{E}$ by $\mathscr{F}$}. For any morphism $T \to \mathscr{M}$ of perfect stacks, its $T$-points are given by pairs $(\zeta_T, \mathscr{L})$, where $\mathscr{L} \in \mathrm{Pic}(T)$ and $\zeta_T$ is an exact triangle in $\Dqc(X_T)$,
\[
\mathscr{F}_T \otimes f_T^*\mathscr{L} \to \mathscr{G}_T \to \mathscr{E}_T,
\]
which is \emph{fiberwise nonsplit over $T$}: that is, for every point $t \in T$, the induced connecting morphism
\(
v|_{X_t} \colon \mathscr{E}|_{X_t} \to \mathscr{F}|_{X_t}[1]
\)
defines a nonzero class in the extension group:
\[
v|_{X_t} \neq 0 \in \mathrm{Ext}^1_{X_t}(\mathscr{E}|_{X_t}, \mathscr{F}|_{X_t}).
\]
There is a \emph{universal exact triangle} in $\Dqc(X_P)$,
\[
\zeta_{\mathrm{univ}} \colon \mathscr{F}_{P} \otimes f_P^*\mathscr{O}_P(1) \to \mathscr{G}_{P} \to \mathscr{E}_{P},
\]
which is fiberwise nonsplit over $P$. In particular, $\mathbf{PExt}_f^1(\mathscr{E},\mathscr{F})$ provides a derived generalization of the classical scheme $P$ considered in \cite[\S 4]{Lag83}.
\end{example}

The following is immediate from the theory of derived projectivizations (\cite[\S 4.2]{J22a}):

\begin{lemma}
\label{lem:PExt}
The derived stack $\pi \colon \mathbf{PExt}^i_f(\mathscr{E}, \mathscr{F}) \to \mathscr{M}$ defined in Definition~\ref{defn:PExt^i} satisfies the following properties:
\begin{enumerate}
    \item \textbf{(Base change)} Formation of $\mathbf{PExt}^i_f(\mathscr{E}, \mathscr{F})$ commutes with arbitrary base change in $\mathscr{M}$.
    
    \item \textbf{(Representability)} The projection $\pi$ is proper,  locally almost of finite presentation, and a relative derived affine scheme over $\mathscr{M}$.
    
    \item \textbf{(Points over fields)} For any morphism $\Spec \kappa \to \mathscr{M}$, where $\kappa$ is a field, let $X_\kappa \to \operatorname{Spec}(\kappa)$ denote the base change of $\mathscr{X} \to \mathscr{M}$. The space of $\kappa$-points is homotopy equivalent to
    \[
    \mathbf{PExt}^i_f(\mathscr{E}, \mathscr{F})(\kappa) \simeq \left( \Ext^i_{X_\kappa}\big(\mathscr{E}|_{X_\kappa}, \mathscr{F}|_{X_\kappa}\big) \setminus \{0\} \right) / \kappa^\times,
    \]
    that is, the set of nonzero $i$-th extension classes of $\mathscr{E}|_{X_\kappa}$ by $\mathscr{F}|_{X_\kappa}$, modulo scaling.
    
    \item \textbf{(Euler exact triangle)} The projection $\pi$ admits a connective relative cotangent complex $\mathbb{L}_{\mathbf{PExt}^i_f(\mathscr{E}, \mathscr{F})/\mathscr{M}}$, which fits into the Euler exact triangle in $\Dqc(\mathbf{PExt}^i_f(\mathscr{E}, \mathscr{F}))$:
    \[
    \mathbb{L}_{\mathbf{PExt}^i_f(\mathscr{E}, \mathscr{F})/\mathscr{M}} \otimes \pi^* \mathscr{O}(1) \to \pi^* \mathcal{Q}_{f}(\mathscr{E}, \mathscr{F}[i]) \to \mathscr{O}(1).
    \]
    
    \item \textbf{(Finiteness and smoothness)} If $\shQ_{f}(\sE,\sF)$ is perfect of Tor-amplitude in $[i, i+k]$ for some $k \geq 0$, then $\mathbb{L}_{\mathbf{PExt}^i_f(\mathscr{E}, \mathscr{F})/\mathscr{M}}$ is perfect of Tor-amplitude in $[0, k]$. In particular, $\pi$ is smooth if $k = 0$, and quasi-smooth if $k \leq 1$.
\end{enumerate}
\end{lemma}


\subsection{Brill--Noether situations}
\label{sec:BN}
The construction of this section is particularly usueful in Brill--Noether situations (see \cite{BCJ}). Specifically, let $X/S$ be a flat, proper morphism of schemes, and let $M_i \to S$ be moduli stacks associated to $X/S$. Set $X_i := X \times_S M_i$ for $i=1,2$, and let $\mathscr{U}_i \in \Dqc(X_i)$ denote the universal object, assumed to be almost perfect, and relatively perfect over $M_i$ (which holds automatically if $\mathscr{U}_i$ is flat over $M_i$). For any $T \to S$, we write $T$-points of $M_i$ as $[F_{T\,i}] \in M_i(T)$, where $F_{T\,i} \in \Dqc(X_{T})$ is the pullback of $\mathscr{U}_i$ along $X \times_S M_i \to M_i$.

In this setting, consider
\[
f \colon \mathscr{X} := X \times_S M_1 \times_S M_2 \to \mathscr{M} := M_1 \times_S M_2.
\]
Let $\mathrm{pr}_{ij}$ denote the projection from $\mathscr{X}$ to the product of its $i$-th and $j$-th factors. Let
\[
\mathscr{E} = \mathrm{pr}_{12}^*(\mathscr{U}_1), \qquad \mathscr{F} = \mathrm{pr}_{13}^*(\mathscr{U}_2).
\]
By Corollary~\ref{cor:Ext:fiberwise}, the complex $\mathcal{Q}_f(\mathscr{E},\mathscr{F})$ is $i$-connective if and only if
\begin{itemize}
  \item $\Ext_{X_s}^j(F_1, F_2) = 0$ for all $j < i$, for every $s \in S$, $[F_1] \in M_1(s)$, and $[F_2] \in M_2(s)$.
\end{itemize}
Here, $X_s$ denotes the fiber of the flat family $X \to S$ over $s$, and, as noted above, $F_i \in \Dqc(X_s)$ corresponds to the $s$-point $[F_i] \in M_i(s)$ for $i = 1, 2$. 

 Thus,  the moduli stack $\mathbf{PExt}_f^i(\mathscr{E},\mathscr{F})$ parametrizes, for each $T \to S$, the data $(F_{1\,T}, F_{2\,T}, \mathscr{L}_T, v)$, where $[F_{i\,T}] \in M_i(T)$ are $T$-points of $M_i$ corresponding to $F_{T\,i} \in \Dqc(X_{T})$, $i=1,2$, $\mathscr{L}_T \in \mathrm{Pic}(T)$ is a line bundle, and $v \colon F_{1\,T} \to F_{2\,T} \otimes f_T^* \mathscr{L}_T$ is a morphism in $\Dqc(X_T)$ that is fiberwise nonzero.

\begin{example}[Hecke modifications]
\label{eg:Hecke}
In the above context, we note that $M = X$ serves as the moduli space parametrizing points of $X/S$, with universal object $\mathscr{U} = \mathscr{O}_{\Delta_{X/S}} \in \Dqc(X \times_S X)$.
\begin{enumerate}
    \item Suppose $M_2 = X \to S$, with $\mathscr{U}_2 = \mathscr{O}_{\Delta_{X/S}}$. Then $\mathcal{Q}_f(\mathscr{E}, \mathscr{F})$ is connective if $\mathscr{U}_1$ is connective. In this case, the moduli stack $\mathbf{PHom}_f(\mathscr{E},\mathscr{F}) \to M_1 \times_S X$ parametrizes, for each $T \to S$, data $(F_{1\,T}, x, \mathscr{L}_T, v)$, where $[F_{1\,T}] \in M_1(T)$ is a $T$-point of $M_1$, $x \colon T \to X_T$ is a section of $f_T \colon X_T \to T$, $\mathscr{L}_T \in \mathrm{Pic}(T)$ a line bundle, and
    \[
        v \colon F_{1\,T} \to x_*(\mathscr{L}_T)
    \]
    is a morphism in $\Dqc(X_T)$ that is fiberwise nonzero over $T$.

    \item Suppose $M_1 = X \to S$, with $\mathscr{U}_1 = \mathscr{O}_{\Delta_{X/S}}$. Then $\mathcal{Q}_f(\mathscr{E}, \mathscr{F})$ is $1$-connective if
    \begin{itemize}
        \item $\Ext_{X_s}^j(\mathscr{O}_x, F_2) = 0$ for all $j \leq 1$, $s \in S$, $x \in X_s(s)$, and $[F_2] \in M_2(s)$.
    \end{itemize}
    In this case, the moduli stack $\mathbf{PExt}^1_f(\mathscr{E},\mathscr{F}) \to X \times_S M_2$ parametrizes, for each $T \to S$, data $(x, F_{2\,T}, \mathscr{L}_T, w)$, where $[F_{2\,T}] \in M_2(T)$ is a $T$-point of $M_2$, $x \colon T \to X_T$ is a section of $f_T$, $\mathscr{L}_T \in \mathrm{Pic}(T)$ is a line bundle, and
    \[
        w \colon x_*(\mathscr{L}_T^\vee) \to F_{2\,T}[1]
    \]
    is a fiberwise nonzero morphism in $\Dqc(X_T)$. The morphism $w$ canonically corresponds to an extension sequence on $X_T$ that is fiberwise non-split over $T$:
    \[
        F_{2\,T} \to E_T \to x_*(\mathscr{L}_T^\vee).
    \]
\end{enumerate}
Comparing with \cite[Propositions 8.7 \& 8.8]{J22b}, these constructions generalize the Hecke correspondence moduli considered by Negu{\c{t}} in \cite{Neg19, Neg22} and by the author in \cite[\S 8]{J22b} to the setting of non-smooth families and non-perfect complexes.
\end{example}

\subsection{Moduli of systems of extensions $\mathbf{GExt}_{f}^i$}
\label{subsec:GExt}
In the context of Definition~\ref{defn:PExt^i}, let $\ell \geq 0$ be an integer, we can also consider the derived moduli stack
\[
\mathbf{GExt}_{f}^i(\mathscr{E},\mathscr{F};\ell) := \mathbf{Grass}_{\mathscr{M}}\left(\mathcal{Q}_{f}(\mathscr{E},\mathscr{F}[i]); \ell\right) \to \mathscr{M}.
\]
Here, for a connective complex $\mathcal{G}$ on $\mathscr{M}$, $\mathbf{Grass}_{\mathscr{M}}(\mathcal{G}; \ell) \to \mathscr{M}$ denotes the \emph{derived Grassmannian} constructed in~\cite{J22b}, which assigns to each morphism $\eta \colon T \to \mathscr{M}$ from a derived affine scheme $T$ the space of pairs $(u, \mathscr{V})$, with $\mathscr{V}$ a vector bundle on $T$ of rank~$\ell$, and $u : \eta^* \mathcal{G} \to \mathscr{V}$ a morphism in $\Dqc(T)$ that is surjective on~$\pi_0$.
Note that by construction:
\[
\mathbf{GExt}_{f}^i(\mathscr{E},\mathscr{F};0) = \mathscr{M}, \qquad
\mathbf{GExt}_{f}^i(\mathscr{E},\mathscr{F};1) = \mathbf{PExt}_{f}^i(\mathscr{E},\mathscr{F}).
\]By an argument analogous to that of Lemma~\ref{lem:PExt:T-points}, we obtain:

\begin{lemma}
\label{lem:GExt:T-points}
Let $\eta \colon T \to \mathscr{M}$ be a morphism from a perfect stack. The space of  $T$-points of $\mathbf{GExt}_f^i(\mathscr{E},\mathscr{F};\ell)$ over $\eta$ consists of pairs $(v, \mathscr{V})$, where $\mathscr{V}$ is a vector bundle of rank $\ell$ on $T$, and
\[
v \colon \mathscr{E}_T \to \mathscr{F}_T[i] \otimes f_T^*\mathscr{V}
\]
is a homomorphism in $\Dqc(X_T)$ such that, for every $t \in T$, the induced map
\(
v|_{X_t} \colon \mathscr{E}|_{X_t} \to \mathscr{F}|_{X_t}^{\oplus \ell}[i]
\)
on the fiber $X_t = \mathscr{X} \times_T \{t\}$ yields an injective map
\[
\kappa(t)^{\oplus \ell} \hookrightarrow \mathrm{Ext}_{X_t}^i\big(\mathscr{E}|_{X_t}, \mathscr{F}|_{X_t}\big).
\]
\end{lemma}

The following is immediate from the theory of derived Grassmannians (\cite[\S 4.1]{J22b}):

\begin{lemma}
The derived stack $\pi \colon \mathbf{GExt}^i_f(\mathscr{E}, \mathscr{F};\ell) \to \mathscr{M}$ satisfies the following properties:
\begin{enumerate}
    \item \textbf{(Base change)} Formation of $\mathbf{GExt}^i_f(\mathscr{E}, \mathscr{F};\ell)$ commutes with any base change in $\mathscr{M}$.
    
    \item \textbf{(Representability)} The projection $\pi$ is proper, locally almost of finite presentation, and a relative derived affine scheme over $\mathscr{M}$.
    
    \item \textbf{(Points over fields)} For any morphism $\Spec \kappa \to \mathscr{M}$, where $\kappa$ is a field, let $X_\kappa \to \operatorname{Spec}(\kappa)$ denote the base change of $X \to \mathscr{M}$. The space of $\kappa$-points is homotopy equivalent to
    \[
    \mathbf{GExt}^i_f(\mathscr{E}, \mathscr{F};\ell)(\kappa) \simeq \operatorname{Hom}^{\mathrm{rk}=\ell}\left(\kappa^{\oplus \ell}, \Ext^i_{X_\kappa}\big(\mathscr{E}|_{X_\kappa}, \mathscr{F}|_{X_\kappa}\big) \right) / \mathrm{GL}_\ell(\kappa^\times),
    \]
    that is, the set of $\ell$-dimension $\kappa$-linear subspaces of $ \Ext^i_{X_\kappa}\big(\mathscr{E}|_{X_\kappa}, \mathscr{F}|_{X_\kappa}\big)$.   
    \item \textbf{(Universal exact triangle)} The projection $\pi$ admits a connective relative cotangent complex $\mathbb{L}_{\mathbf{GExt}^i_f(\mathscr{E}, \mathscr{F}; \ell)/\mathscr{M}}$, which fits an exact triangle in $\Dqc(\mathbf{GExt}^i_f(\mathscr{E}, \mathscr{F};\ell))$:
    \[
    \mathbb{L}_{\mathbf{GExt}^i_f(\mathscr{E}, \mathscr{F};\ell)/\mathscr{M}} \otimes \pi^* \mathscr{Q} \to \pi^* \mathcal{Q}_{f}(\mathscr{E}, \mathscr{F}[i]) \to \mathscr{Q},
    \]
    where $\mathscr{Q}$ denote the universal rank $\ell$ quotient bundle.
    \item \textbf{(Finiteness and smoothness)} If $\shQ_{f}(\sE,\sF)$ is perfect of Tor-amplitude in $[i, i+k]$ for some $k \geq 0$, then $\mathbb{L}_{\mathbf{GExt}^i_f(\mathscr{E}, \mathscr{F};\ell)/\mathscr{M}}$ is perfect of Tor-amplitude in $[0, k]$. In particular, $\pi$ is smooth if $k = 0$, and quasi-smooth if $k \leq 1$.
\end{enumerate}
\end{lemma}

\subsection{Semiorthogonal decompositions for moduli of extensions}
\label{subsec:SOD:Ext}
Let $\mathscr{E}, \mathscr{F} \in \Dqc(\mathscr{X})$ with $\mathscr{E}$ perfect and $\mathscr{F}$ relatively perfect over $\mathscr{M}$. 
Suppose there exists $i \in \mathbb{Z}$ such that $\mathcal{Q}_f(\mathscr{E},\mathscr{F})$ has Tor-amplitude in $[i, i+1]$. By Corollary~\ref{cor:Ext:fiberwise}, this is equivalent to requiring that
\begin{itemize}
    \item For every point $s \in \mathscr{M}$, $\Ext_{X_s}^j(\mathscr{E}|_{X_s}, \mathscr{F}|_{X_s}) = 0$ for $j \neq i, i+1$.
\end{itemize}
The rank of $\mathcal{Q}_{f}(\mathscr{E},\mathscr{F})$ then defines a locally constant function on $\mathscr{M}$; restricting to a connected component, we may assume it is constant, say $r$. Again by Corollary~\ref{cor:Ext:fiberwise}, this is equivalent to: 
\[
r = \dim_{\kappa(s)} \Ext_{X_s}^{i}(\mathscr{E}|_{X_s}, \mathscr{F}|_{X_s}) - \dim_{\kappa(s)} \Ext_{X_s}^{i+1}(\mathscr{E}|_{X_s}, \mathscr{F}|_{X_s}) \quad \text{for all $s \in \mathscr{M}$.}
\]
In this case, $\mathcal{Q}_f(\mathscr{E},\mathscr{F}[i])$ has Tor-amplitude in $[0,1]$. By Serre duality (Proposition~\ref{prop:Serre:Q}),
\[
\mathcal{Q}_f(\mathscr{E},\mathscr{F}[i])^\vee[1] \simeq \mathcal{Q}_f(\mathscr{F}[i], \mathscr{E} \otimes \omega_{f}^\bullet)[1] \simeq \mathcal{Q}_f(\mathscr{F}, \mathscr{E} \otimes \omega_{f}^\bullet[-i-1]).
\]

Applying the main result of \cite[Theorem~7.5]{J22a} (see also \cite[Theorem~3.4]{JL18}) to the perfect complex $\mathcal{Q}_f(\mathscr{E},\mathscr{F}[i])$ of Tor-amplitude $[0,1]$ over $\mathscr{M}$, we obtain:

\begin{theorem}
\label{thm:SOD:PExt}
In the above setting, the natural projections 
	\[\mathbf{PExt}_f^i(\mathscr{E}, \mathscr{F}) \to \mathscr{M} \quad \text{and} \quad \mathbf{PHom}_{f}(\mathscr{F}, \mathscr{E} \otimes \omega_f^\bullet[-i-1]) \to \mathscr{M}\] 
are quasi-smooth.
Moreover, let $\mathrm{D}$ denote $\Dqc$, $\mathrm{D}^\mathrm{-}_{\mathrm{coh}}$, $\mathrm{D}^\mathrm{b}_{\mathrm{coh}}$, or $\mathrm{Perf}$, then:
\begin{enumerate}
    \item If $r > 0$, there is an $\mathscr{M}$-linear semiorthogonal decomposition
    \[
    \mathrm{D}\left(\mathbf{PExt}_f^i(\mathscr{E}, \mathscr{F})\right) = \left\langle \mathrm{D}\left(\mathbf{PHom}_{f}(\mathscr{F}, \mathscr{E} \otimes \omega_f^\bullet[-i-1])\right),\ \text{$r$ copies of } \mathrm{D}(\mathscr{M}) \right\rangle.
    \]
    \item If $r = 0$, there is an $\mathscr{M}$-linear derived equivalence
    \[
    \mathrm{D}\left(\mathbf{PExt}_f^i(\mathscr{E}, \mathscr{F})\right) \simeq \mathrm{D}\left(\mathbf{PHom}_{f}(\mathscr{F}, \mathscr{E} \otimes \omega_f^\bullet[-i-1])\right).
    \]
    \item If $r < 0$, there is an $\mathscr{M}$-linear semiorthogonal decomposition
    \[
    \mathrm{D}\left(\mathbf{PHom}_{f}(\mathscr{F}, \mathscr{E} \otimes \omega_f^\bullet[-i-1])\right) = \left\langle \mathrm{D}\left(\mathbf{PExt}_f^i(\mathscr{E}, \mathscr{F})\right),\ \text{$(-r)$ copies of } \mathrm{D}(\mathscr{M}) \right\rangle.
    \]
\end{enumerate}
\end{theorem}

This theorem can be viewed as a categorification of the fiberwise Riemann--Roch formula:
\[
r =  \dim_{\kappa(s)} \Ext_{X_s}^{i}(\mathscr{E}|_{X_s}, \mathscr{F}|_{X_s}) -  \dim_{\kappa(s)} \Ext_{X_s}^{i+1}(\mathscr{E}|_{X_s}, \mathscr{F}|_{X_s}).
\]  

For the moduli stacks $\mathbf{GExt}_{f}^i$, applying the main result of \cite{J23}, we further obtain:

\begin{theorem}
\label{thm:SOD:GExt}
In the above setting, the natural projections
\[
\mathbf{GExt}_f^i(\mathscr{E}, \mathscr{F}; \ell) \to \mathscr{M}
\quad \text{and} \quad
\mathbf{GExt}_f^0(\mathscr{F}, \mathscr{E} \otimes \omega_f^\bullet[-i-1]; \ell') \to \mathscr{M}
\]
are quasi-smooth for all integers $\ell, \ell'$. 
Moreover, suppose $\mathscr{M}$ is a derived stack over a field of characteristic zero, and let $\mathrm{D}$ denote $\Dqc$, $\mathrm{D}^-_{\mathrm{coh}}$, $\mathrm{D}^\mathrm{b}_{\mathrm{coh}}$, or $\mathrm{Perf}$. Then:
\begin{enumerate}
    \item If $r > 0$, for any $\ell \geq 0$, there is an $\mathscr{M}$-linear semiorthogonal decomposition
    \[
    \mathrm{D}\left(\mathbf{GExt}_f^i(\mathscr{E}, \mathscr{F};\ell)\right) 
    = \left \langle  \text{$\binom{r}{j}$ copies of } \mathrm{D}\left(\mathbf{GExt}_{f}^0(\mathscr{F}, \mathscr{E} \otimes \omega_f^\bullet[-i-1]; \ell-j) \right) \mid 0 \leq j \leq \ell \right\rangle.
    \]
    \item If $r = 0$, for any $\ell \geq 0$, there is an $\mathscr{M}$-linear equivalence
    \[
    \mathrm{D}\left(\mathbf{GExt}_f^i(\mathscr{E}, \mathscr{F};\ell)\right) \simeq \mathrm{D}\left(\mathbf{GExt}_{f}^0(\mathscr{F}, \mathscr{E} \otimes \omega_f^\bullet[-i-1]; \ell)\right).
    \]
    \item If $r < 0$, for any $\ell \geq 0$, there is an $\mathscr{M}$-linear semiorthogonal decomposition
    \[
    \mathrm{D}\left(\mathbf{GExt}_{f}^0(\mathscr{F}, \mathscr{E} \otimes \omega_f^\bullet[-i-1]; \ell)\right) 
    = \left\langle  \text{$\binom{-r}{j}$ copies of } \mathrm{D}\left(\mathbf{GExt}_f^i(\mathscr{E}, \mathscr{F};\ell-j)\right)  \mid 0 \leq j \leq \ell  \right\rangle.
    \]
\end{enumerate}
\end{theorem}

The results above have a wide range of applications.
Before focusing on families of curves in the next section, we briefly illustrate the scope of Theorem~\ref{thm:SOD:PExt} with examples involving pairs, surfaces, and Calabi--Yau families; Theorem~\ref{thm:SOD:GExt} can be applied similarly in these cases.

\begin{example}[Pairs]
\label{eg:stable.pairs}
Let $\mathscr{X} \to \mathscr{M}$ be a flat, proper family of classical varieties. Suppose that for each $s \in \mathscr{M}$,  $\mathscr{F}|_{X_s}$ is a coherent sheaf on $X_s$ with support of dimension at most one, i.e., $\mathscr{F}|_{X_s} \in \mathrm{Coh}_{\le 1}(X_s)$. Then the conditions of the theorems are satisfied for $i=0$ and $\mathscr{E} = \mathscr{O}_{\mathscr{X}}$.

In this case, the derived moduli stack $\mathbf{PHom}_f(\mathscr{O}_{\mathscr{X}}, \mathscr{F})$ parametrizes equivalence classes of pairs $(F, v)$, where $F = \mathscr{F}|_{X_s}$ and $v \colon \mathscr{O}_{X_s} \to F$ is a nonzero section. This moduli space is closely related to the stable pair moduli studied by Pandharipande--Thomas~\cite{PT09, PT}, Bridgeland~\cite{Bri11}, and Bryan--Steinberg~\cite{BS16, DPSV23}. In the case of curves, it also connects to Thaddeus pairs~\cite{Tha}, as discussed further in \S\ref{subsec:coh.sys}.
\end{example}

\begin{example}[Surface case]
Let $\mathscr{X} \to \mathscr{M}$ be a flat family of proper Cohen--Macaulay surfaces, and set $\omega := \omega_f^\bullet[-2]$ for the relative dualizing sheaf. For $i=0$, Theorem \ref{thm:SOD:PExt} implies:
\begin{align*}
&\mathrm{D}\left(\mathbf{PHom}_f(\mathscr{E}, \mathscr{F})\right) =
\left\langle 
    \mathrm{D}\left(\mathbf{PExt}^1_{f}(\mathscr{F},\, \mathscr{E} \otimes \omega)\right),\ 
    \text{$r$ copies of } \mathrm{D}(\mathscr{M}) 
\right\rangle 
&\quad \text{if $r \ge 0$}, \\
&\mathrm{D}\left(\mathbf{PExt}^1_f(\mathscr{F},\, \mathscr{E} \otimes \omega)\right) =
\left\langle 
    \mathrm{D}\left(\mathbf{PHom}_{f}(\mathscr{E}, \mathscr{F})\right),\ 
    \text{$(-r)$ copies of } \mathrm{D}(\mathscr{M}) 
\right\rangle 
&\quad \text{if $r \le 0$}.
\end{align*}
A similar statement holds for $i=1$, and these results extend to the stacks $\mathbf{GExt}_f^i$ in characteristic zero by Theorem~\ref{thm:SOD:GExt}.
These results further generalize to two-dimensional categories, including Calabi--Yau 2 (CY2) categories as in~\cite{Kuz10, AT14}; see also~\cite{Che21} and~\cite[\S 5.3]{J19}.
\end{example}

\begin{example}[Calabi--Yau families]
Let $f \colon \mathscr{X} \to \mathscr{M}$ be a flat, proper, representable morphism. We say $f$ is a \emph{Calabi--Yau family of dimension $d \ge 0$} if all its geometric fibers are (classical) Gorenstein varieties of dimension $d$, and there exists a line bundle $\mathscr{L}$ on $\mathscr{M}$ such that $\omega_f^\bullet \simeq f^*\mathscr{L}[d]$. In this case, each fiber $X_s$ is a Calabi--Yau variety of dimension $d$, with dualizing complex $\omega_{X_s}^\bullet \simeq \mathscr{O}_{X_s}[d]$. Theorem \ref{thm:SOD:PExt} then gives:
\begin{align*}
  &\mathrm{D}\left(\mathbf{PExt}_f^i(\mathscr{E}, \mathscr{F})\right) = 
    \left\langle 
      \mathrm{D}\left(\mathbf{PExt}_{f}^{d-i-1}(\mathscr{F}, \mathscr{E}) \right),\ 
      \text{$r$ copies of } \mathrm{D}(\mathscr{M}) 
    \right\rangle 
  && \text{if $r \ge 0$}, \\
  &\mathrm{D}\left(\mathbf{PExt}_f^{d-i-1}(\mathscr{F}, \mathscr{E})\right) = 
    \left\langle 
      \mathrm{D}\left(\mathbf{PExt}_{f}^{i}(\mathscr{E}, \mathscr{F}) \right),\ 
      \text{$(-r)$ copies of } \mathrm{D}(\mathscr{M}) 
    \right\rangle 
  && \text{if $r \le 0$}.
\end{align*}
\end{example}

We further remark that Brill--Noether settings, as described in \S\ref{sec:BN}, can be incorporated into the above examples, and the framework of this section is especially useful in these contexts.

 \section{Abel maps for integral curves}
In this section, we focus on families $X/S$ of integral curves and prove the main results outlined in the introduction using the theory developed in previous sections.

\subsection{Abel maps in general}
\label{subsec:Abel:general}
We begin with a slightly more general framework. This setup applies, in particular, to the case $\mathscr{M} = \overline{\mathrm{Pic}}_{X/S}^d$, the compactified Picard scheme, as well as its suitable derived enhancement, and their higher-rank variants (see \S \ref{subsec:coh.sys}). 

Let $X \to S$ be a flat, locally projective, and finitely presented morphism of schemes, whose geometric fibers are integral curves of arithmetic genus $p_a$. Suppose we have a derived $S$-stack $\mathscr{M} \to S$ and a Cartesian diagram
\[
\begin{tikzcd}
    \mathscr{X} \arrow[r] \arrow[d, "f"'] & X \arrow[d] \\
   \mathscr{M} \arrow[r] & S.
\end{tikzcd}
\]
Assume there exists a ``universal'' object $\mathscr{J} \in \Dqc(\mathscr{X})$, perfect relative to $\mathscr{M}$, such that for each fiber $X_s$, we have $H^j(X_s, \mathscr{J}|_{X_s}) = 0$ for $j \neq 0,1$. (This condition is automatic if each $\mathscr{J}|_{X_s}$ is a coherent sheaf.)

By Corollary~\ref{cor:Ext:fiberwise}, the complex $\mathcal{Q}_f(\mathscr{O}_{\mathscr{X}}, \mathscr{J})$ is perfect of Tor-amplitude $[0,1]$, with locally constant rank. By passing to connected components of $\mathscr{M}$, we may assume the rank is constant, say $r$. Again by Corollary~\ref{cor:Ext:fiberwise}, this means
\begin{equation}\label{eqn:Abel:r}
r = \dim_{\kappa(s)} H^0(X_s, \mathscr{J}|_{X_s}) - \dim_{\kappa(s)} H^1(X_s, \mathscr{J}|_{X_s}) \quad \text{for all } s \in \mathscr{M}.
\end{equation}

Let $\omega  := \omega_{f}^\bullet[-1]$ denote the relative dualizing sheaf. For any $s \in \mathscr{M}$, its restriction $\omega|_{X_s}$ is the absolute dualizing sheaf of the curve $X_s$ (\cite[Proposition 6.6.3.1]{SAG}).

\begin{definition}[Abel Maps]
\label{defn:abel:LinSys}
We define the Abel maps in this setting as
\begin{align*}
    &\mathbf{A}^{\mathscr{J}} \colon \quad \mathbf{LinSys}_{X/S}^{\mathscr{J}} := \mathbf{PHom}_f(\mathscr{O}_{\mathscr{X}}, \mathscr{J}) \longrightarrow \mathscr{M}, \\
    &\mathbf{A}^{\mathscr{J}}_{\omega} \colon \quad \mathbf{Quot}_{\omega/X/S}^{\mathscr{J}} := \mathbf{PHom}_f(\mathscr{J}, \omega) \longrightarrow \mathscr{M},
\end{align*}
where $\mathbf{PHom}_f$ is the derived moduli of $0$-th extensions as in Definition~\ref{defn:PExt^i} and Example~\ref{eg:PHom}.
\end{definition}

Concretely, by Lemma~\ref{lem:PExt:T-points}, for any map $\eta \colon T \to \mathscr{M}$ from a derived affine scheme $T$ (with $f_T \colon X_T \to T$ the base change of $f$), the $T$-points $\mathbf{LinSys}_{X/S}^{\mathscr{J}}(T)$ and $\mathbf{Quot}_{\omega/X/S}^{\mathscr{J}}(T)$ over $\eta$ parametrize pairs $(u, \mathscr{L})$ and $(v, \mathscr{L}')$, where $\mathscr{L}, \mathscr{L}' \in \mathrm{Pic}(T)$, and
\[
u \colon \mathscr{O}_T \to \mathscr{J}_T \otimes f_T^*\mathscr{L}, \qquad
v \colon \mathscr{J}_T \otimes f_T^*(\mathscr{L}'^\vee) \to \omega_T
\]
are homomorphisms in $\Dqc(X_T)$ which are fiberwise nonzero over $T$ (see Definition~\ref{defn:fiberwise.nonzero}).

In particular, for any point $s \colon \operatorname{Spec} \kappa \to \mathscr{M}$ (with $X_s$ the fiber of $f$ over $s$), the fibers of the classical truncations of the Abel maps are the projective spaces
\begin{equation}\label{eqn:fibers:abel}
\mathbb{P}\left(H^0(X_s, \mathscr{J}|_{X_s})^\vee\right)
\quad \text{and} \quad
\mathbb{P}\left(\operatorname{Hom}_{X_s}(\mathscr{J}|_{X_s}, \omega_{X_s})^\vee\right) \simeq \mathbb{P}\left(H^1(X_s, \mathscr{J}|_{X_s})\right),
\end{equation}
where we use Grothendieck's convention: for a $\kappa$-vector space $V$, the $\kappa$-points of $\mathbb{P}(V^\vee)$ are $(V \setminus \{0\})/\kappa^\times$. These are derived generalizations of the usual Abel maps \cite[(8.2)]{AK80}.

We summarize the main properties of the Abel maps in this context as follows:

\begin{theorem}
\label{thm:SOD:abel:LinSys}
In the setting above, let $\mathrm{D}$ denote any of $\Dqc$, $\mathrm{D}^-_{\mathrm{coh}}$, $\mathrm{D}^\mathrm{b}_{\mathrm{coh}}$, or $\mathrm{Perf}$. Let $r$ denote the rank of $\mathcal{Q}_f(\mathscr{O}_{X}, \mathscr{J})$, equivalently given by formula~\eqref{eqn:Abel:r}. Then:
\begin{enumerate}
    \item The Abel maps $\mathbf{A}^{\mathscr{J}}$ and $\mathbf{A}^{\mathscr{J}}_{\omega}$ are quasi-smooth of relative virtual dimension $r-1$ and $-r-1$, respectively. 
    \item
    \begin{enumerate}
        \item If $r > 0$, there is an $\mathscr{M}$-linear semiorthogonal decomposition
        \[
        \mathrm{D}\big( \mathbf{LinSys}_{X/S}^{\mathscr{J}} \big) = \left\langle \mathrm{D}\big( \mathbf{Quot}_{\omega/X/S}^{\mathscr{J}} \big),\ \text{$r$ copies of } \mathrm{D}(\mathscr{M}) \right\rangle.
        \]
        \item If $r = 0$, there is an $\mathscr{M}$-linear derived equivalence
        \[
        \mathrm{D}\big( \mathbf{LinSys}_{X/S}^{\mathscr{J}} \big) \simeq \mathrm{D}\big( \mathbf{Quot}_{\omega/X/S}^{\mathscr{J}} \big).
        \]
        \item If $r < 0$, there is an $\mathscr{M}$-linear semiorthogonal decomposition
        \[
        \mathrm{D}\big( \mathbf{Quot}_{\omega/X/S}^{\mathscr{J}} \big) = \left\langle \mathrm{D}\big( \mathbf{LinSys}_{X/S}^{\mathscr{J}} \big),\ \text{$(-r)$ copies of } \mathrm{D}(\mathscr{M}) \right\rangle.
        \]
    \end{enumerate}
    \item If $\omega$ is locally truncated (for example, if $\mathscr{M}$ is classical, or if $\mathscr{O}_{\mathscr{M}}$ is locally truncated), set $\mathbb{D}_\omega(\mathscr{J}) := \sHom_X(\mathscr{J}, \omega)$. Then there is a canonical isomorphism of derived stacks
    \[
    \mathbf{LinSys}_{X/S}^{\mathscr{J}} \simeq \mathbf{Quot}_{\omega/X/S}^{\mathbb{D}_\omega(\mathscr{J})}.
    \]
\end{enumerate}
\end{theorem}

\begin{proof}
Part (1) follows from Lemma~\ref{lem:PExt}.(5).  
Part (2) is a consequence of Theorem~\ref{thm:SOD:PExt}.  
Part (3) follows from the coherent duality $\mathcal{Q}_f(\mathscr{O}_{X}, \mathscr{J}) \simeq \mathcal{Q}_f(\mathbb{D}_\omega(\mathscr{J}), \omega)$ of Proposition~\ref{prop:coh:duality}.
\end{proof}
 
In characteristic zero, applying Theorem~\ref{thm:SOD:GExt}, we also obtain:

\begin{theorem}
\label{thm:SOD:abel:GLinSys}
In the above setting, assume further that $X \to S$ are schemes over a field of characteristic zero. For any $\ell \geq 0$, the generalized Abel maps
\begin{align*}
&\mathbf{A}^{\mathscr{J}, \ell} \colon \quad \mathbf{G}_{\ell}\mathbf{LinSys}_{X/S}^{\mathscr{J}} := \mathbf{GExt}_f^0(\mathscr{O}_{X}, \mathscr{J}; \ell) \to \mathscr{M}, \\
&\mathbf{A}^{\mathscr{J}, \ell}_{\omega} \colon \quad \mathbf{G}_{\ell}\mathbf{Quot}_{\omega/X/S}^{\mathscr{J}} := \mathbf{GExt}_f^0(\mathscr{J}, \omega; \ell) \to \mathscr{M}
\end{align*}
are quasi-smooth. Moreover, there are $\mathscr{M}$-linear semiorthogonal decompositions:
\begin{align*}
    \mathrm{D}\big(\mathbf{G}_{\ell}\mathbf{LinSys}_{X/S}^{\mathscr{J}}\big) &= \left\langle \text{$\binom{r}{j}$ copies of } \mathrm{D}\big(\mathbf{G}_{\ell-j}\mathbf{Quot}_{\omega/X/S}^{\mathscr{J}}\big) \mid 0 \leq j \leq \ell \right\rangle \quad \text{if $r \geq 0$}, \\
    \mathrm{D}\big(\mathbf{G}_{\ell}\mathbf{Quot}_{\omega/X/S}^{\mathscr{J}}\big) &= \left\langle \text{$\binom{-r}{j}$ copies of } \mathrm{D}\big(\mathbf{G}_{\ell-j}\mathbf{LinSys}_{X/S}^{\mathscr{J}}\big) \mid 0 \leq j \leq \ell \right\rangle \quad \text{if $r \leq 0$}.
\end{align*}
Furthermore, if $\omega$ is locally truncated, then for all $\ell \geq 0$ there is a canonical isomorphism:
\[
    \mathbf{G}_{\ell}\mathbf{LinSys}_{X/S}^{\mathscr{J}} \simeq \mathbf{G}_{\ell}\mathbf{Quot}_{\omega/X/S}^{\mathbb{D}_{\omega}(\mathscr{J})}.
\]
\end{theorem}

    \subsection{Abel maps for compactified Picard schemes}
  \label{subsec:Abel:cPic}

We now focus on the case $\mathscr{M} = \overline{\mathrm{Pic}}_{X/S}^d$, the compactified Picard scheme from the introduction. While our arguments extend to suitable derived enhancements of $\overline{\mathrm{Pic}}_{X/S}^d$, we leave such generalizations for future work.

\subsubsection{General cases}
\label{subsec:Abel:integral.curves}
In the setting of the Introduction, by \cite[Thm.~3.4(3)]{AK79}, \cite[(8.5)]{AK80}, the compactified Picard scheme $\mathscr{M} = \overline{\mathrm{Pic}}_{X/S}^d$ is a projective $S$-scheme equipped with a universal coherent sheaf $\mathscr{J} \in \mathrm{Coh}(\mathscr{X})$, which is flat over $\mathscr{M}$ and whose restriction $\mathscr{J}|_{X_s}$ to each fiber $X_s$, where $s \in \mathscr{M}$, is a rank $1$, torsion-free sheaf of degree $d$. Here, the degree of $\mathscr{J}|_{X_s}$ is defined by
\[
\deg(\mathscr{J}|_{X_s}) := \chi(X_s, \mathscr{J}|_{X_s}) - \chi(X_s, \mathscr{O}_{X_s}),
\]
where $\chi(X_s, \mathscr{F}) := \dim_{\kappa(s)} H^0(X_s, \mathscr{F}) - \dim_{\kappa(s)} H^1(X_s, \mathscr{F})$. Thus, the degree $d$ and the integer $r$ in formula~\eqref{eqn:Abel:r} are related by the Riemann--Roch formula: 
\[
r  = 1 - p_a + d
\]

To distinguish the case for each degree $d \in \mathbb{Z}$, we write $\mathscr{J}_d$ for the universal sheaf $\mathscr{J}$ in the case $\mathscr{M}=\overline{\mathrm{Pic}}_{X/S}^d$, and denote $f$ by $f_d\colon \mathscr{X} \to \overline{\mathrm{Pic}}_{X/S}^d$.

For each $d$, we consider the Abel map $\mathbf{A}_{\omega}^{\mathscr{J}_d} \colon \mathbf{Quot}_{\omega/X/S}^{\mathscr{J}_d} \to \overline{\mathrm{Pic}}_{X/S}^d$ of Definition~\ref{defn:abel:LinSys}, which in this context we denote by
\[
\mathbf{A}_\omega^d \colon \mathbf{Quot}_{\omega/X/S}^{2p_a - 2 - d} \to \overline{\mathrm{Pic}}_{X/S}^d.
\]

By \cite[\href{https://stacks.math.columbia.edu/tag/0C0Z}{Tag~0C0Z}]{stacks-project}, for the Cohen--Macaulay morphism $f_d\colon \mathscr{X} \to \mathscr{M} = \overline{\mathrm{Pic}}_{X/S}^d$, the relative dualizing sheaf $\omega = f_d^!(\mathscr{O}_{\mathscr{M}})[-1]$ is a coherent $\mathscr{O}_{\mathscr{X}}$-module, and restricts to the absolute dualizing sheaf $\omega_{X_s}$ on each fiber $X_s$ (\cite[\href{https://stacks.math.columbia.edu/tag/0BZZ}{Tag~0BZZ}, \href{https://stacks.math.columbia.edu/tag/0AWT}{Tag~0AWT}]{stacks-project}). Since $\mathscr{J}|_{X_s}$ is a maximal Cohen--Macaulay sheaf of rank $1$, \cite[Theorem~3.3.10]{BH} yields $\sExt_{X_s}^j(\mathscr{J}|_{X_s}, \omega_{X_s}) = 0$ for $j \neq 0$, and $\sExt_{X_s}^0(\mathscr{J}|_{X_s}, \omega_{X_s})$ is a rank $1$, torsion-free sheaf of degree $2p_a-2-d$. Moreover, the natural map $\mathscr{J}|_{X_s} \to \sExt_{X_s}^0(\sExt_{X_s}^0(\mathscr{J}|_{X_s}, \omega_{X_s}), \omega_{X_s})$ is an isomorphism. 
By \cite[Thm.~1.10]{AK80}, $\mathbb{D}_\omega(\mathscr{J}_d) \simeq \sExt_{\mathscr{X}}^0(\mathscr{J}_d, \omega)$ is a $\mathscr{M}$-flat, coherent $\mathscr{O}_\mathscr{X}$-module, compatible with base change. Applying the same argument with $\mathscr{J}_d$ replaced by $\mathbb{D}_\omega(\mathscr{J}_d)$, the natural map $\mathscr{J}_d \to \mathbb{D}_\omega(\mathbb{D}_\omega(\mathscr{J}_d))$ is a homomorphism of $\mathscr{M}$-flat coherent sheaves, compatible with base change and restricting to an isomorphism on each fiber. Thus, this map is an isomorphism.

Thus, for each $d$, there is a unique morphism
\[
\iota\colon \overline{\mathrm{Pic}}_{X/S}^d \to \overline{\mathrm{Pic}}_{X/S}^{2p_a-2-d}
\]
such that $\iota^*\mathscr{J}_{2p_a-2-d} \simeq \mathbb{D}_\omega(\mathscr{J}_d)$. The isomorphism $\mathscr{J}_d \xrightarrow{\simeq} \mathbb{D}_\omega(\mathbb{D}_\omega(\mathscr{J}_d))$ shows that $\iota$ defines an involutive automorphism of the whole compactified Picard scheme $\overline{\mathrm{Pic}}_{X/S} = \coprod_{d \in \mathbb{Z}} \overline{\mathrm{Pic}}_{X/S}^d$ over $S$, interchanging components of degree $d$ and $2p_a-2-d$.

By Proposition~\ref{prop:coh:duality}, together with the base-change properties of Proposition~\ref{thm:Q}~\eqref{thm:Q-2} and the compatibility of $\omega$ with base change, we have canonical isomorphisms
\[
\mathcal{Q}_{f_d}(\mathscr{O}_{\mathscr{X}}, \mathscr{J}_d) \simeq \mathcal{Q}_{f_d}(\mathbb{D}(\mathscr{J}_d), \omega) \simeq \iota^* \left( \mathcal{Q}_{f_{2p_a-2-d}}(\mathscr{J}_{2p_a-2-d}, \omega) \right).
\]
Since the formation of $\mathbf{PHom}_f$ commutes with base change (Lemma~\ref{lem:PExt}), the map $\iota$ induces an isomorphism
\[
\iota^*\colon \mathbf{Quot}_{\omega/X/S}^d \xrightarrow{\sim} \mathbf{LinSys}_{X/S}^{\mathscr{J}_d}
\]
which fits into the following commutative diagram:
\begin{equation*}
\begin{tikzcd}[row sep=large, column sep=large]
    \mathbf{Quot}_{\omega/X/S}^d \arrow[r, "\iota^*", "\sim"'] \arrow[d, swap, "\mathbf{A}_\omega^{2p_a - 2 - d}"] & \mathbf{LinSys}_{X/S}^{\mathscr{J}_d} \arrow[d, "\mathbf{A}^{\mathscr{J}_d}"] \\
    \overline{\mathrm{Pic}}_{X/S}^{2p_a - 2 - d} \arrow[r, "\iota", "\sim"'] & \overline{\mathrm{Pic}}_{X/S}^d
\end{tikzcd}
\end{equation*}
We define the (modified) Abel map in this context by the composition
\[
\mathbf{A}^d := \iota \circ \mathbf{A}_\omega^{2p_a - 2 - d} \simeq \mathbf{A}^{\mathscr{J}_d} \circ \iota^* \colon \mathbf{Quot}_{\omega/X/S}^d \to \overline{\mathrm{Pic}}_{X/S}^d.
\]
The classical truncations of the fibers of $\mathbf{A}^d$ and $\mathbf{A}_\omega^d$ are again the projective spaces \eqref {eqn:fibers:abel}.

The main properties of the Abel maps in this context are summarized as follows:

\begin{theorem}
\label{thm:SOD:curve}
In the above setting:
\begin{enumerate}
    \item For any $d \in \mathbb{Z}$, the Abel maps $\mathbf{A}^d$ and $\mathbf{A}_\omega^d$ are quasi-smooth of relative virtual dimensions $d-p_a$ and $p_a - 2 - d$, respectively. 
    \item For each $d$, the classical truncation of $\mathbf{Quot}_{\omega/X/S}^d$ is the classical Quot scheme $\mathrm{Quot}_{\omega/X/S}^d$ of \cite[(8.2)]{AK80}, and the classical truncations of $\mathbf{A}^d$ and $\mathbf{A}_\omega^d$ recover the classical Abel maps $A^d$ and $A_\omega^d$ described in the introduction\footnote{Note that our $A_\omega^d$ corresponds to $\mathscr{A}_\omega^{2p_a-2-d}$ in \cite[(8.2)]{AK80}, while $A^d$ corresponds to $\iota \circ \mathscr{A}_\omega^d$.}.
    \item The derived scheme $\mathbf{Quot}_{\omega/X/S}^d$ is empty for $d < 0$.
    \item The Abel map $\mathbf{A}^d$ is surjective if and only if $d \ge p_a$. In this case, for any type of derived category $\mathrm{D} \in \{\Dqc, \mathrm{D}^-_{\mathrm{coh}}, \mathrm{D}^\mathrm{b}_{\mathrm{coh}}, \mathrm{Perf}\}$, there is a semiorthogonal decomposition:
    \[
    \mathrm{D}(\mathbf{Quot}_{\omega/X/S}^d) = \left\langle \mathrm{D}(\mathbf{Quot}_{\omega/X/S}^{2p_a-2-d}),\ \text{$(1-p_a+d)$ copies of } \mathrm{D}(\overline{\mathrm{Pic}}_{X/S}^d) \right\rangle.
    \]
    \item If $d = p_a - 1$, then there is a ``flopping" derived equivalence
    \[
    \mathrm{D}(\mathbf{Quot}_{\omega/X/S}^d) \simeq \mathrm{D}(\mathbf{Quot}_{\omega/X/S}^{2p_a-2-d}).
    \]
    \item The Abel map $\mathbf{A}^d$ is smooth if and only if $d \ge 2p_a - 1$. In this case, $\mathbf{Quot}_{\omega/X/S}^d = \mathrm{Quot}_{\omega/X/S}^d$ is classical, $\mathbf{A}^d$ is a projective bundle with fiber $\mathbb{P}^{d-p_a}$, and $\mathbf{Quot}_{\omega/X/S}^{2p_a-2-d} = \emptyset$, so the decomposition in (4) reduces to Orlov's projective bundle formula \cite{Orlov92}:
    \[
    \mathrm{D}(\mathrm{Quot}_{\omega/X/S}^d) = \left\langle \text{$(1-p_a+d)$ copies of } \mathrm{D}(\overline{\mathrm{Pic}}_{X/S}^d) \right\rangle.
    \]
\end{enumerate}
\end{theorem}

\begin{proof}
Part (1) follows from Lemma~\ref{lem:PExt}.(5).

For (2), it suffices to consider the Abel map $\mathbf{A}_\omega^d$, as $\mathbf{A}^d$ is obtained by composing with the involutive isomorphism $\iota$. By Remark~\ref{rmk:H(I,F)}, the truncation $\pi_0\mathcal{Q}_f(\mathscr{J},\omega)$ coincides with Altman--Kleiman's module $H(\mathscr{J}, \omega)$ \cite[(1.1)]{AK80}. Thus, by \cite[Proposition 4.21]{J23}, the classical truncation of the derived projectivization $\mathbb{P}(\mathcal{Q}_f(\mathscr{J}, \omega))$ is isomorphic to the classical projectivization of $H(\mathscr{J}, \omega)$ over $\mathscr{M}$, and the result follows by comparing Definition~\ref{defn:abel:LinSys} with \cite[Lemma (5.17)]{AK80}.

Part (3) follows from the vanishing $H^0(X_s, \mathscr{J}|_{X_s}) = 0$ for $\deg(\mathscr{J}|_{X_s}) < 0$ on each fiber $X_s$ (see \cite[(3.5)(iii.f)]{AK80}), part (2), and the description of Abel map fibers \eqref{eqn:fibers:abel}.

For (4), surjectivity follows from part (2) and \cite[Theorem 8.4(ii)]{AK80}. The remainder of (4), as well as (5), follow from Theorem~\ref{thm:SOD:PExt}.

For (6), the “only if” direction follows from \cite[Theorem 8.4(v)]{AK80}. If $d \ge 2p_a - 2$, then $H^1(X_s, \mathscr{J}|_{X_s}) = 0$ on all fibers (see \cite[(3.5)(iii.g)]{AK80}), so by Corollary~\ref{cor:Ext:fiberwise}, $\mathcal{Q}_f(\mathscr{O}_{\mathscr{X}}, \mathscr{J})$ is a vector bundle of rank $1-p_a+d$ over $\mathscr{M} = \overline{\mathrm{Pic}}_{X/S}^d$. The result follows.
\end{proof}

\begin{remark}
The derived equivalence in (5) is of ``flopping'' type. For example, if $X/S$ is a smooth projective curve $C$ over a field, \cite[Theorem~4.1]{JL18} shows that this equivalence arises from the $\Theta$-flop, and that composing the flopping equivalences yields a highly nontrivial autoequivalence -- namely, a spherical twist along the $\Theta$-divisor in the sense of \cite{ST, AL}.

More generally, the diagram
\[
\mathbf{Quot}_{\omega/X/S}^{p_a - 1}
\xrightarrow{\mathbf{A}^{p_a-1}}
\overline{\mathrm{Pic}}_{X/S}^{p_a-1}
\xleftarrow{\mathbf{A}^{p_a-1}_{\omega}}
\mathbf{Quot}_{\omega/X/S}^{p_a - 1}
\]
should be viewed as a derived analogue of the $\Theta$-flop, with the derived equivalence in (5) providing a confirmation of the DK conjecture in this context. We expect that a similar ``flop--flop = twist'' phenomenon, as in \cite[Corollary~4.4]{JL18}, holds in this more general setting.
\end{remark}

\subsubsection{Gorenstein case}
\label{subsec:Abel:Gorenstein}
Suppose all fibers (or equivalently, all geometric fibers) $X_s$ of $X/S$ are Gorenstein. Then the relative dualizing sheaf $\omega$ is invertible (\cite[\href{https://stacks.math.columbia.edu/tag/0C08}{Tag~0C08}]{stacks-project}). For any $d \in \mathbb{Z}$, define the derived Hilbert scheme and its unmodified Abel map by (see Definition \ref{defn:abel:LinSys}):
\[
\widetilde{\mathbf{A}}_{-d} \colon \mathbf{Hilb}_{X/S}^d := \mathbf{PHom}_{f_{-d}}(\mathscr{J}_{-d}, \mathscr{O}_{\mathscr{X}}) \to \overline{\mathrm{Pic}}_{X/S}^{~-d}.
\]
Tensoring with $\omega$ yields a Cartesian diagram in which the horizontal arrows are isomorphisms:
\begin{equation*}
\begin{tikzcd}[row sep=large, column sep=large]
    \mathbf{Hilb}_{X/S}^d \arrow[r, "\otimes \omega", "\sim"'] \arrow[d, swap, "\widetilde{\mathbf{A}}_{-d}"] &  \mathbf{Quot}_{\omega/X/S}^d  \arrow[d, "\mathbf{A}_{\omega}^{2p_a -2 -d}"] \\
    \overline{\mathrm{Pic}}_{X/S}^{~-d} \arrow[r, "\otimes \omega", "\sim"'] & \overline{\mathrm{Pic}}_{X/S}^{2p_a-2-d}
\end{tikzcd}
\end{equation*}

Composing with the isomorphism $\otimes \omega$, we obtain Abel maps in this Gorenstein setting:
\[
\mathbf{A}_d := \mathbf{A}^d \circ (\otimes \omega)\colon \mathbf{Hilb}_{X/S}^d \to \overline{\mathrm{Pic}}_{X/S}^d, \qquad
\mathbf{A}^{\omega}_d := \mathbf{A}_{\omega}^d \circ (\otimes \omega) \colon \mathbf{Hilb}_{X/S}^{2p_a-2-d} \to \overline{\mathrm{Pic}}_{X/S}^{d}.
\]
Here, the positions of the upper and lower indices of Abel maps are switched to distinguish from the Quot scheme case. By construction, these Abel maps are identified with the derived projectivizations of the complexes $\mathcal{Q}_{f_d}(\mathscr{O}_{\mathscr{X}}, \mathscr{J}_d)$ and $\mathcal{Q}_{f_d}(\mathscr{J}_d, \omega)$ over $\overline{\mathrm{Pic}}_{X/S}^{d}$, respectively.

For a point $[J] \in \overline{\mathrm{Pic}}^d(X_s)$, the classical truncations of the fibers of $\mathbf{A}_d$ and $\mathbf{A}_d^\omega$ over $[J]$ are the projective spaces $\mathbb{P}(H^0(X_s, J)^\vee)$ and $\mathbb{P}(H^0(X_s, \sExt^0_{X_s}(J, \omega_{X_s}))^\vee)$, corresponding to the complete linear systems $|J|$ and $|\sExt^0(J, \omega_{X_s})|$ of \emph{generalized divisors} for the rank $1$, torsion-free sheaves $J$ and $\sExt^0_{X_s}(J, \omega_{X_s})$ in the sense of Hartshorne \cite{Har86, Har94}.

\begin{corollary}
\label{cor:Abel:Gorenstein}
In the above situation, for any $d$, the classical truncation of $\mathbf{Hilb}_{X/S}^d$ is the classical Hilbert scheme $\mathrm{Hilb}_{X/S}^d$ parametrizing flat families of zero-dimensional subschemes of length $d$ on the fibers of $X/S$. Moreover, all statements of Theorem~\ref{thm:SOD:curve} remain true with all $\mathbf{Quot}_{\omega/X/S}^d$ replaced by $\mathbf{Hilb}_{X/S}^d$, $\mathrm{Quot}_{\omega/X/S}^d$ by $\mathrm{Hilb}_{X/S}^d$, and $\mathbf{A}^d$, $\mathbf{A}_\omega^d$ by $\mathbf{A}_d$, $\mathbf{A}_d^\omega$, respectively.
\end{corollary}

The derived Hilbert schemes $\mathbf{Hilb}_{X/S}^d$ are usually not classical. 
However, in a notable special case,  the derived and classical Hilbert schemes coincide: 

\subsubsection{Locally planar case}
\label{subsec:Abel:planar}

We say that $X/S$ is a family of \emph{locally planar curves} if, locally, $X/S$ admits a closed embedding into a smooth quasi-projective family of surfaces $Y$ over $S$. In particular, every fiber of $X/S$ is Gorenstein.

\begin{lemma}[Locally planar curves] 
\label{lem:planar:classical}
If $X/S$ is a family of locally planar curves, then the derived Hilbert schemes are classical: $\mathbf{Hilb}_{X/S}^d = \mathrm{Hilb}_{X/S}^d$ for all $d \ge 0$.
\end{lemma}
\begin{proof}
By \cite[Theorem (6)]{AIK77}, $\mathrm{Hilb}_{X/S}^d \to S$ is a flat, locally complete intersection morphism of relative dimension $d$ (see also \cite{Reg80, BGS81}). 
Using the isomorphism $\otimes \mathscr{O}_{X/S}(m \cdot x): \overline{\mathrm{Pic}}_{X/S}^d \xrightarrow{\sim} \overline{\mathrm{Pic}}_{X/S}^{d+m}$ for any $m \in \mathbb{Z}$, and the fact that $A^{d+m}$ is smooth for $m \gg 0$ (Theorem~\ref{thm:SOD:curve}(6)), it follows that $\overline{\mathrm{Pic}}_{X/S}^d \to S$ is a flat, locally complete intersection morphism of relative dimension $p_a$ for all $d$. Together with Theorem~\ref{thm:SOD:curve}(1), this shows that $\mathbf{Hilb}_{X/S}^d \to S$ is quasi-smooth of virtual relative dimension $d$. The claim now follows from the following standard lemma.
\end{proof}

\begin{lemma}
Let $Z \to S$ be a quasi-smooth map of derived schemes of virtual relative dimension $d$. If $S$ is classical and $Z_{\mathrm{cl}}$ is flat over $S$ of relative dimension $d$, then $Z = Z_{\mathrm{cl}}$.
\end{lemma}

\begin{proof}
The statement is local on $S$, so we may assume $S = \operatorname{Spec} R$ and $Z$ is the derived zero locus of a sequence $f_1, \ldots, f_m \in A = R[x_1, \ldots, x_n]$, with $d = n - m$ (see \cite[pp.~45]{DAG}). It suffices to show that $f_1, \ldots, f_m$ is a regular sequence at each point $[\mathfrak{p}]\in Z_{\mathrm{cl}} = \operatorname{Spec}(A/(f_1,\ldots, f_m))$.

Let $[\mathfrak{q}] \in \operatorname{Spec} R$ be the image of $[\mathfrak{p}]$. By flatness, the fiber of $Z_{\mathrm{cl}}$ over $[\mathfrak{q}]$ at $[\mathfrak{p}]$ is defined by the images $\overline{f}_1, \ldots, \overline{f}_m$ in $A_{\mathfrak{p}} / \mathfrak{q} A_{\mathfrak{p}}$. Since $A_{\mathfrak{p}} / \mathfrak{q} A_{\mathfrak{p}}$ is a smooth $\kappa(\mathfrak{q})$-algebra of dimension $n$  and the fiber of $Z_{\mathrm{cl}}$ has dimension $n-m$, the sequence $\overline{f}_1, \ldots, \overline{f}_m$ is regular in $A_{\mathfrak{p}} / \mathfrak{q} A_{\mathfrak{p}}$. By \cite[\href{https://stacks.math.columbia.edu/tag/0470}{Tag~0470}]{stacks-project}, it follows that $f_1, \ldots, f_m$ is a regular sequence in $A_{\mathfrak{p}}$. Thus, $Z$ is classical.
\end{proof}

Therefore, in the locally planar case, all statements of Theorem~\ref{thm:SOD:curve} hold for the classical schemes, with $\mathbf{Quot}_{\omega/X/S}^d$ replaced by $\mathrm{Hilb}_{X/S}^d$, and $\mathbf{A}^d$, $\mathbf{A}_\omega^d$ by the classical Abel maps $A_d$ and $A_d^\omega$, respectively. In particular, the semiorthogonal decomposition of Theorem~\ref{thm:SOD:curve}(4)\&(5) specializes as follows: for each $d \ge p_a - 1$, there is a semiorthogonal decomposition:
\[
\mathrm{D}(\Hilb_{X/S}^d) = \left\langle \mathrm{D}(\Hilb_{X/S}^{2p_a-2-d}),\ \text{$(1-p_a+d)$ copies of } \mathrm{D}(\overline{\mathrm{Pic}}_{X/S}^d) \right\rangle.
\]
This greatly extends the known semiorthogonal decompositions for symmetric powers of smooth projective curves over a field $\kappa$ due to Toda~\cite[Corollary~5.11]{Tod2} (see also \cite[Corollary~3.10]{JL18}, \cite{BK19}, and \cite[Example~3.33]{J22a}). 

Moreover, in the case of a smooth projective curve $C$ over $\mathbb{C}$, this formula is optimal: by a result of Lin~\cite{Lin}, the derived category of $\Hilb_{C}^d = \operatorname{Sym}^d(C)$ admits no nontrivial semiorthogonal decompositions for $d \le g(C) - 1$. Motivated by this, we conjecture that in the locally planar case, the derived category $\mathrm{D}(\Hilb_{X/S}^d)$ does not admit any nontrivial $S$-linear semiorthogonal decomposition when $d \le p_a - 1$.

\subsection{Moduli of linear series of generalized divisors}
\label{subsec:Abel:Gdl}
We now return to the general setting of a family $X/S$ of integral curves as in \S\ref{subsec:Abel:integral.curves}, and consider Grassmannian-type generalizations of the Abel maps, as in Theorem~\ref{thm:SOD:GExt}. 

Given an integer $\ell \geq -1$, we define the moduli stack of linear series of generalized divisors of degree $d$ and dimension $\ell$, together with the associated generalized Abel map, by
\[
\mathbf{A}^{d,\ell} := \mathbf{A}^{\mathscr{J}_d, \ell} \colon  \qquad \overline{\mathbf{G}}_{d, X/S}^\ell := \mathbf{G}_{\ell+1}\mathbf{LinSys}_{X/S}^{\mathscr{J}_d} \longrightarrow \overline{\mathrm{Pic}}_{X/S}^d.
\]
Let $\iota$ denote the involution of $\overline{\mathrm{Pic}}_{X/S}$ as in \S\ref{subsec:Abel:integral.curves}. The generalized $\omega$-Abel map is defined as
\[
\mathbf{A}_{\omega}^{d, \ell} \colon \quad
\overline{\mathbf{G}}_{2p_a - 2 - d, X/S}^\ell
\xrightarrow[\sim]{\quad\iota^*\quad}
\mathbf{G}_{\ell+1}\mathbf{Quot}_{\omega/X/S}^{\mathscr{J}_d}
\xrightarrow{~\mathbf{A}_{\omega}^{\mathscr{J}, \ell+1}~}
\overline{\mathrm{Pic}}_{X/S}^d.
\]
The classical truncations of the fibers of $\mathbf{A}_{d,\ell}$ and $\mathbf{A}_{\omega}^{d, \ell}$ are Grassmannian varieties:
\[
\mathrm{Grass}_{\ell+1}\left(H^0(X_s, \mathscr{J}|_{X_s})^\vee\right)
\quad \text{and} \quad
\mathrm{Grass}_{\ell+1}\left(\operatorname{Ext}^0_{X_s}(\mathscr{J}|_{X_s}, \omega_{X_s})^\vee\right),
\]
respectively. Equivalently, these are the Grassmannians parametrizing $\ell$-dimensional projective linear subspaces of \eqref{eqn:fibers:abel}, i.e., 
$\mathbb{P}\left(H^0(X_s, \mathscr{J}|_{X_s})^\vee\right)$ and $\mathbb{P}\left(\operatorname{Ext}^0_{X_s}(\mathscr{J}|_{X_s}, \omega_{X_s})^\vee\right)$, respectively.

If the curve $X_s$ is smooth, then the classical truncation of the fiber of $\overline{\mathbf{G}}_{d, X/S}^\ell$ over a point $s \in S$ recovers the variety $G_d^\ell(X_s)$ of $\ell$-dimensional linear series $g_d^\ell$ of degree $d$ on $X_s$, as studied in~\cite[Chapter~IV]{ACGH}. In particular, the morphism $\overline{\mathbf{G}}_{d, X/S}^\ell \to S$, restricted over the locus where the fibers $X_s$ are smooth, is a derived enhancement of the family of classical varieties of linear series $G_d^\ell$ of \cite[Chapter~IV]{ACGH}, studied in \cite[\S 4.3]{J23}.

More generally, over points $s \in S$ where $X_s$ is Gorenstein, the classical truncation of $\overline{\mathbf{G}}_{d, X/S}^\ell$ over $s$ is the variety $\mathrm{G}_{d}^{\ell}$ of $\ell$-dimensional projective linear subspaces of the complete linear system $|\mathscr{J}|_{X_s}|$ of generalized divisors. This provides a natural generalization of the classical variety of linear series to the setting of generalized divisors in the sense of~\cite{Har86, Har94}.

An immediate consequence of Theorem~\ref{thm:SOD:GExt} is the following:

\begin{corollary}
\label{cor:SOD:Gdl}
\begin{enumerate}
	\item The Abel maps $\mathbf{A}^{d,\ell}$ and  $\mathbf{A}_{\omega}^{d, \ell}$ are quasi-smooth for all $d$ and $\ell$.
	\item Suppose that $S$ is a $\mathbb{Q}$-scheme (i.e., of characteristic zero), and let $\D$ be one of $\Dqc$, $\mathrm{D}^-_{\mathrm{coh}}$, $\mathrm{D}^\mathrm{b}_{\mathrm{coh}}$, or $\mathrm{Perf}$. Then for all $d \geq p_a-1$ and $\ell \geq -1$, there is a semiorthogonal decomposition:
\[
\D\left(\overline{\mathbf{G}}_{d, X/S}^\ell\right)
=
\left\langle
\text{$\binom{1-p_a+d}{j}$ copies of } \D\left(\overline{\mathbf{G}}_{2p_a-2-d, X/S}^{\ell-j}\right)
\,\Big|\, 0 \leq j \leq \ell+1
\right\rangle.
\]
\end{enumerate}
\end{corollary}

This result generalizes the semiorthogonal decompositions of the author \cite[Corollary 4.5]{J23} from families of smooth curves to families of integral curves. See also \cite{Tod23} for the case of a general smooth curve $C/\mathbb{C}$. Furthermore, when $C$ is a general smooth curve, the moduli stack $\overline{\bG}_{d, C/\CC}^{\ell}$ recovers the classical variety $G_d^\ell(C)$. In this context, Lin and Yu~\cite{LY21} proved that the derived categories $\D(G_{2g-2-d}^{r-i}(C))$ are indecomposable for all $0 \leq i \leq r+1$.

\subsection{Moduli of coherent systems on curves}
\label{subsec:coh.sys}
We return to the general setting of~\ref{subsec:Abel:general}, where $X \to S$ is a flat, locally projective, and finitely presented morphism of schemes, whose geometric fibers are integral curves of arithmetic genus $p_a$.  Assume that $S$ is locally Noetherian.

 For any integers $d \in \mathbb{Z}$ and $\rho > 0$, consider the moduli stack of torsion-free sheaves of rank $\rho$ and degree $d$ on the fibers of $X/S$:
\[
\mathscr{M} := \underline{\mathrm{Coh}}_{X/S}^{\mathrm{tf}}(\rho, d) \to S.
\]
For any $S$-scheme $T$, the $T$-points of $\mathscr{M}$ consist of $T$-flat coherent sheaves $\mathscr{E}_T \in \mathrm{Coh}(X_T)$ such that, for each $t \in T$, the restriction $\mathscr{E}|_{X_t}$ is a torsion-free sheaf of rank $\rho$ and degree $d$ on $X_t$. Here, the degree is defined as $\deg(\mathscr{E}|_{X_t}) := \chi(X_t, \mathscr{E}|_{X_t}) - \chi(X_t, \mathscr{O}_{X_t})$. Thus, the degree $d$ and the integer $r$ in~\eqref{eqn:Abel:r} are related via the Riemann--Roch formula:
\[
r = (1 - p_a)\rho + d.
\]

Let $\mathscr{E} \in \mathrm{Coh}(\mathscr{X})$ denote the universal sheaf.
 For $\ell \ge 0$, we consider the following derived stacks and their ``Abel'' maps:
\begin{align*}
&\mathbf{A}^{\rho, d, \ell} \colon \quad \mathbf{G}_{X/S}(\rho, d, \ell) := \mathbf{GExt}_f^0(\mathscr{O}_{X}, \mathscr{E}; \ell) \to \mathscr{M}, \\
&\mathbf{A}^{\rho, d, \ell}_\omega \colon \quad \mathbf{G}^\omega_{X/S}(\rho, d, \ell) := \mathbf{GExt}_f^0(\mathscr{E}, \omega; \ell) \to \mathscr{M}.
\end{align*}

By Lemma~\ref{lem:GExt:T-points}, for any $S$-scheme $T$, the $T$-points of $\mathbf{G}_{X/S}(\rho, d, \ell)$ and $\mathbf{G}^\omega_{X/S}(\rho, d, \ell)$ parametrize triples $(v, \mathscr{V}, \mathscr{E}_T)$ and $(w, \mathscr{W}, \mathscr{F}_T)$, where $\mathscr{V}$ and $\mathscr{W}$ are vector bundles over $T$ of rank $\ell$, $\mathscr{E}_T$ and $\mathscr{F}_T$ are $T$-flat, torsion-free sheaves of rank $\rho$ and degree $d$ on fibers of $X_T/T$, and 
\[
v \colon f_T^*(\mathscr{V}_T^\vee) \to \mathscr{E}_T, \qquad
w \colon \mathscr{F}_T \to \omega_{X_T/T} \otimes f_T^*(\mathscr{W}_T)
\]
are homomorphisms in $\Dqc(X_T)$ such that, for each $t \in T$, the induced maps on the fiber $X_t$ yield injective maps
\[
\kappa(t)^{\oplus \ell} \hookrightarrow H^0(X_t, \mathscr{E}|_{X_t}) \quad \text{and} \quad \kappa(t)^{\oplus \ell} \hookrightarrow \operatorname{Hom}_{X_t}(\mathscr{F}|_{X_t}, \omega_{X_t}) \simeq H^1(X_t, \mathscr{E}|_{X_t})^\vee.
\]
In particular, the functor of points over a field $s\colon \Spec\kappa(s) \to S$ assigns to each such $s$ pairs $(E, V)$ and $(F, W)$, where $E$ and $F$ are torsion-free sheaves of rank $\rho$ and degree $d$ on $X_s = X \times_S \{s\}$, and $V \subseteq H^0(X_s, E)$ and $W \subseteq H^1(X_s, F)^\vee$ are linear subspaces of dimension~$\ell$.

Thus, the classical truncations of these moduli spaces recover the moduli of coherent systems (or Brill--Noether pairs), extensively studied by many authors, see, for example, \cite{KN95}. 

In the case $\ell = 1$, we use the notation:
\[
\mathbf{Pair}_{X/S}(\rho, d) := \mathbf{PHom}_f(\mathscr{O}_X, \mathscr{E}) \to \mathscr{M}, \qquad
\mathbf{Pair}_{X/S}^{\omega}(\rho, d) := \mathbf{PHom}_f(\mathscr{E}, \omega) \to \mathscr{M}.
\]
In this case, their functors of points over a field classify Thaddeus pairs \cite{Tha, KT} $(E, v)$ and $(F, w)$ respectively, where $E$ and $F$ are torsion-free sheaves on $X_s$, and $v \in H^0(X_s, E)$, $w \in \operatorname{Hom}_{X_s}(F, \omega)$ are nonzero sections, considered up to scalar multiplication. 

Theorems~\ref{thm:SOD:abel:LinSys} and~\ref{thm:SOD:GExt} yield the following:

\begin{corollary}
\label{cor:SOD:coh.sys}
Let $\D$ be one of $\Dqc$, $\mathrm{D}^-_{\mathrm{coh}}$, $\mathrm{D}^\mathrm{b}_{\mathrm{coh}}$, or $\mathrm{Perf}$.
\begin{enumerate}
    \item The maps $\mathbf{A}^{\rho, d, \ell}$ and $\mathbf{A}^{\rho, d, \ell}_\omega$ are quasi-smooth for all $\rho>0$, $d \in \mathbb{Z}$ and $\ell \ge 0$.
    \item For all $d \geq (p_a-1)\rho$, there is an $\mathscr{M}$-linear semiorthogonal decomposition:
    \[
    \mathrm{D}(\mathbf{Pair}_{X/S}(\rho, d)) =
    \left\langle
        \mathrm{D}(\mathbf{Pair}_{X/S}^{\omega}(\rho, d)),\ 
        \text{$((1-p_a)\rho+d)$ copies of } \mathrm{D}(\mathscr{M})
    \right\rangle.
    \]
    \item If $S$ is a $\mathbb{Q}$-scheme, then for all $d \geq (p_a-1)\rho$ and $\ell \geq 0$, there is an $\mathscr{M}$-linear semiorthogonal decomposition:
    \[
    \D\left(\mathbf{G}_{X/S}(\rho, d, \ell)\right) =
    \left\langle
        \text{$\binom{(1-p_a)\rho+d}{j}$ copies of }
        \D\left(\mathbf{G}^\omega_{X/S}(\rho, d, \ell-j)\right)
        \,\Big|\, 0 \leq j \leq \ell
    \right\rangle.
    \]
\end{enumerate}
\end{corollary}

Specializing to the case where $X/S$ is a smooth projective curve $C/\mathbb{C}$, $\mathrm{gcd}(\rho, d) = 1$ and $\ell = 1$, the semiorthogonal decomposition in~(1), when restricted to the open substack $\mathcal{B}un_{C}^{\mathrm{st}}(\rho, d) \subseteq \mathscr{M}$ of stable vector bundles, recovers a result previously obtained by Koseki and Toda~\cite{KT}.

Thus, our result extends the semiorthogonal decompositions for Thaddeus pair moduli spaces for stable vector bundles on smooth curves to the broader context of higher-dimensional coherent systems and torsion-free sheaves on families of integral curves.


\end{document}